\documentclass[a4paper,11pt]{article}
\usepackage[english]{babel}
\usepackage[utf8]{inputenc}

\usepackage{amsmath,amsthm,amssymb}

\usepackage{latexsym,amssymb,amsfonts,amsmath,mathrsfs,float}
\numberwithin{equation}{section}

\usepackage[margin=1.2in]{geometry}

  \newcommand{\R}{\mathbb{R}} % reals
  \newcommand{\N}{\mathbb{N}} % natural numbers
  \newcommand{\st}{:\,} % "such that" to define sets
  \DeclareMathOperator{\rank}{rank}
    
  \newcommand{\eg}{{e.g.}}  

\newcommand{\wtilde}{\widetilde}
 
  \newcommand{\ceil}[1]{\lceil{#1}\rceil}
  \newcommand{\floor}[1]{\lfloor{#1}\rfloor}
  \DeclareMathOperator{\Tr}{\mathsf{Tr}}
 
\newcommand{\beq}{\begin{equation}}
\newcommand{\eeq}{\end{equation}}
\newcommand{\beqn}{\begin{equation*}}
\newcommand{\eeqn}{\end{equation*}}
\newcommand{\beqr}{\begin{eqnarray}}
\newcommand{\eeqr}{\end{eqnarray}}
\newcommand{\beqrn}{\begin{eqnarray*}}
\newcommand{\eeqrn}{\end{eqnarray*}}
\newcommand{\bmline}{\begin{multline}}
\newcommand{\emline}{\end{multline}}
\newcommand{\bmlinen}{\begin{multline*}}
\newcommand{\emlinen}{\end{multline*}}

\newtheorem{defin}{Definition}[section]
\newtheorem{definition}[defin]{Definition}
\newtheorem{proposition}[defin]{Proposition}
\newtheorem{theorem}[defin]{Theorem}
\newtheorem{remark}[defin]{Remark}
\newtheorem{corollary}[defin]{Corollary}
\newtheorem{lemma}[defin]{Lemma}
\newtheorem{example}[defin]{Example}
\newtheorem{conjecture}[defin]{Conjecture}

\makeatletter
\newtheorem*{rep@theorem}{\rep@title}
\newcommand{\newreptheorem}[2]{%
\newenvironment{rep#1}[1]{%
 \def\rep@title{#2 \ref{##1}}%
 \begin{rep@theorem}}%
 {\end{rep@theorem}}}
\makeatother

\newreptheorem{lemma}{Lemma}
\newreptheorem{theorem}{Theorem}
\newreptheorem{corollary}{Corollary}
\newreptheorem{proposition}{Proposition}

\usepackage{enumitem}
\makeatletter
\def\namedlabel#1#2{\begingroup
    #2%
    \def\@currentlabel{#2}%
    \phantomsection\label{#1}\endgroup
}

\usepackage{hyperref}
\usepackage{xcolor,graphicx}

\newcommand{\qindep}{\alpha^{\star}} % proper quantum chromatic number :-)
\newcommand{\qchrom}{\chi^{\star}} % proper quantum chromatic number :-)
\newcommand{\cgpsd}{\mathcal{CS}_{+}}
\newcommand{\cpsd}{\mathcal{S}_{+}}
\newcommand{\dnn}{\mathcal{DNN}}
\newcommand{\cp}{\mathcal{CP}}

\newcommand{\ignore}[1]{}
\newcommand{\Sym}{\text{\rm Sym}}

\newcommand{\M}{{\mathcal M}}
\newcommand{\ncP}{\mathbb{R}\langle \underline{X} \rangle}

\newcommand{\Qnc}{Q_{\text{\rm nc}}}
\newcommand{\Bnc}{B_{\text{\rm nc}}}
\newcommand{\K}{{\mathcal K}}

\newcommand{\cop}{{{\mathcal {COP}}}}
\newcommand{\Knceps}{{\mathcal K}_{\text{\rm nc},\epsilon}}
\newcommand{\Kc}{{\mathcal K}_{\text{\rm c}}}
\newcommand{\cl}{\text{\rm cl}}

\newcommand{\trMcube}{\text{\rm tr}\M^{\text{\rm cube}}_{\text{\rm nc}}}
\newcommand{\trMball}{\text{\rm tr}\M^{\text{\rm ball}}_{\text{\rm nc}}}

\usepackage{authblk}

\begin{document}
\title{Conic approach to quantum graph parameters using linear optimization over the completely positive semidefinite cone}

\author[1,2]{Monique Laurent \thanks{monique@cwi.nl}}
\author[1]{Teresa Piovesan \thanks{piovesan@cwi.nl}}

\affil[1]{Centrum Wiskunde \& Informatica (CWI), Amsterdam, The Netherlands}
\affil[2]{Tilburg University, Tilburg, The Netherlands}

\date{\today}

\small

\maketitle

\begin{abstract}
We investigate the completely positive semidefinite  cone $\cgpsd^n$, a new matrix cone consisting of all $n\times n$ matrices that admit a Gram representation by positive semidefinite matrices (of any size). In particular, we  study  relationships  between this cone and the   completely positive and the doubly nonnegative cone, and between its dual 
cone  and  trace positive non-commutative polynomials. 

We use this new cone  to model quantum analogues of the classical independence and chromatic graph parameters $\alpha(G)$ and $\chi(G)$, which are roughly obtained by allowing variables to be positive semidefinite matrices instead of $0/1$ scalars in the programs defining the classical parameters. We can formulate  these quantum parameters as conic linear programs over the cone $\cgpsd^n$. Using this conic approach we can recover the bounds in terms of the theta number and define further approximations by exploiting the link to trace positive polynomials. 
 \\

\end{abstract}

\section{Introduction}
\subsection{General overview}

Computing the minimum number $\chi(G)$ of colors needed to properly color a graph $G$ and computing the maximum cardinality $\alpha (G)$ of an independent set of vertices in $G$ are two well studied NP-hard problems in combinatorial optimization.
Recently, some analogues of these classical graph parameters have been investigated, namely the parameters $\alpha_q(G)$ and $\chi_q(G)$  in the context of  quantum entanglement in nonlocal games and the parameters $\qindep(G)$ and $\qchrom(G)$ in the context of 
quantum information.  In a nutshell, while the classical parameters can be defined as the optimal values of  integer  programming problems involving $0/1$-valued variables, their quantum analogues are obtained by allowing the  variables to be positive semidefinite matrices (of arbitrary size).

\medskip
To make this precise and simplify our discussion we now focus on the quantum chromatic number $\chi_q(G)$ 
(introduced in \cite{Cameron:2007}).
Given  a graph $G=(V,E)$ and an  integer $t\ge 1$, consider the  following conditions in the variables $x^i_u$ (for $u\in V$ and $i\in [t]=\{1,\ldots,t\}$):
\begin{equation}\label{eq0}
%\begin{array}{ll}
\begin{aligned}
& \sum_{i\in [t]} x^i_u=1 \;\; \forall u\in V,\ \ \ \  x^i_ux^j_u =0 \;\;    \forall u\in V \text{ and } \forall i\ne j\in [t], \\
& x^i_ux^i_v =0 \;\; \forall \{u,v\} \in E \text{ and } \forall i\in [t]. 
\end{aligned}
%%\end{array}
\end{equation}
If the variables are 0/1-valued then these conditions  are encoding the fact  that each vertex of $G$ receives just one out of $t$ possible colors and that two adjacent vertices must receive distinct colors.
Then the chromatic number $\chi(G)$ is equal to the smallest integer $t$ for which the system (\ref{eq0}) admits a $0/1$-valued solution.
On the other hand, if we allow the variables $x^i_u$ to take their values in  $\mathcal S^d_+$ (the cone of $d\times d$ positive semidefinite matrices) for an arbitrary  $d\ge 1$ and, if in the first condition we let 1 denote the identity matrix,
 then the smallest integer $t$ for which the system 
(\ref{eq0}) is feasible defines the quantum parameter $\chi_q(G)$. 
By construction, $ \chi(G) \ge \chi_q(G).$

It is well known that computing the chromatic number $\chi(G)$ is an NP-hard problem and   recently this hardness result has been extended to the quantum chromatic number $\chi_q(G)$ \cite{Ji:2013}. Therefore it is of interest to be able to compute  good approximations for these parameters.
In the classical case, several converging hierarchies of approximations have been proposed for $\chi(G)$ based on semidefinite programming 
(see \cite{Dukanovic:2010,Gvozdenovic:2008}).
They refine the well known bounds based on the theta number of Lov\'asz \cite{Lovasz:1979} and its strengthening by Szegedy \cite{Szegedy:1994}: 
$\chi(G)\ge \vartheta^+(\overline G)\ge \vartheta(\overline G)$.
It was shown recently in \cite{Roberson:2012} that the theta number also bounds  
the quantum chromatic number:
$$\chi(G) \ge \chi_q(G)\ge \vartheta^+(\overline G)\ge \vartheta(\overline G).$$

This raises naturally the question of constructing further semidefinite programming based bounds for $\chi_q(G)$, strengthening the theta number. 
The parameter $\chi_q(G)$ derives from a specific nonlocal game and the general problem of finding approximations to the quantum value of any nonlocal game is a very interesting and difficult one.
Positive results in this direction have been achieved in \cite{Doherty:2008,Navascues:2008} where the authors introduced semidefinite  hierarchies converging to a relaxation of the quantum value of the game. 
%The convergence of these hierarchies
Whether these hierarchies converge  
to the quantum value itself 
is an open problem in mathematical physics, commonly known as Tsirelson's problem.

Here we take a different approach exploiting the particular structure of the quantum graph parameters considered.
The main idea is to reformulate the quantum graph parameter $\chi_q(G)$ as a conic optimization problem over a new matrix cone, the cone $\cgpsd$, that we call the {\em completely positive semidefinite cone}. 
The study of this matrix cone and its use for building approximations is the main contribution of this paper which we explain below in  more detail.

\medskip
Recall that  a  matrix $A\in \mathcal S^n$ is positive semidefinite (psd), i.e.,  $A\in \mathcal S^n_+$, precisely when $A$ admits a Gram representation by {\em  vectors} $x_1,\ldots,x_n\in \R^d$ (for some $d\ge 1$), which means that  $A=(\langle x_i,x_j\rangle )_{i,j=1}^n$.
Moreover, $A$ is {\em completely positive} when it admits such a Gram representation by {\em nonnegative vectors}.
We now call $A$ {\em completely positive semidefinite} when $A$ admits a Gram representation by {\em positive semidefinite matrices}
$x_1,\ldots,x_n\in \mathcal S^d_+$ for some $d\ge 1$ {(where $\langle x_i,x_j\rangle$ denotes the usual trace inner product)}.
We let $\cp^n$ and $\cgpsd^n$ denote, respectively,  the sets of completely positive and completely psd matrices.

The set $\cgpsd^n$ is  a convex cone (Lemma \ref{lemconvex}), but it is not known whether it is   closed. 
A related open question is whether 
any matrix $A$ which admits a Gram representation by {\em infinite} positive semidefinite matrices also admits such a Gram representation by finite ones (see Theorem~\ref{theoclosed}).

It is easy to see that the new cone $\cgpsd^n$ is nested between $\cp^n$ and the {\em doubly nonnegative} cone $\dnn^n$ (consisting of all matrices that are 
both psd and nonnegative): 
$$\cp^n\subseteq \cgpsd^n \subseteq \cl(\cgpsd^n) \subseteq \dnn^n.$$
In what follows we may omit the superscript and write $\cp, \cgpsd, \dnn$  when we do not need to explicitly mention the size $n$ of the matrices.
As is well known,  $\dnn^n=\cp^n$ for $n\le 4$ and strict inclusion holds for $n\ge 5$ \cite{Diananda:1962,Maxfield:1962}.
Frenkel and Weiner \cite{Frenkel:2014} give an example of a $5 \times 5$ matrix which is doubly nonnegative but not completely psd and the authors of  \cite{Fawzi:2014}  give an example of a $5\times 5$ matrix which is completely psd but not completely positive.
We will show that the above mentioned matrix studied by Frenkel and Weiner 
 does not lie in the closure of $\cgpsd^5$,  thus $\cl(\cgpsd^{5}) \subsetneq \dnn^{5}$. Moreover we will present a class of matrices that are doubly nonnegative but not completely psd (Lemma \ref{lemclass}). An ingredient for this is showing that for matrices supported by a cycle, being completely positive is equivalent to being completely psd (Theorem~\ref{lem:oddcycle}).

\medskip
Using the completely psd  cone $\cgpsd$ we  can reformulate the quantum chromatic number $\chi_q(G)$  as a linear optimization problem over affine sections of the cone $\cgpsd$. 
Notice that such a reformulation, combined with the above mentioned result about the NP-hardness of the quantum chromatic number, implies that 
%Combining this with the above mentioned result that the quantum chromatic number is an NP-hard parameter \cite{Ji:2013}, we have that 
linear optimization over affine sections of the completely positive semidefinite cone is also an NP-hard problem.
The idea for reformulating $\chi_q(G)$ is simple and goes as follows:
linearize the system (\ref{eq0}) by introducing a matrix $X$ (defined as the Gram matrix of the psd matrices $x^i_u$), add the condition  $X\in \cgpsd$, and  replace the conditions in (\ref{eq0}) by linear conditions on $X$
(see  Sections \ref{secconealpha}-\ref{secPsi} for details). In this way the whole complexity of the problem is pushed to the cone $\cgpsd$.

By replacing in the conic linear  program defining $\chi_q(G)$ the cone $\cgpsd$ by its closure $\mathrm{cl}(\cgpsd)$, we obtain a new parameter $\wtilde \chi_q(G)$, which satisfies $\chi_q(G) \ge \wtilde \chi_q(G)$. This new parameter $\wtilde \chi_q(G)$ can be equivalently formulated in terms of the dual conic program, since strong duality holds (while we do not know if this is the case for the program defining $\chi_q(G)$). The dual conic program is over the dual cone $\cgpsd^*$. As we explain below, $\cgpsd^*$ can be interpreted in terms of trace positive polynomials, which naturally opens the way to semidefinite programming based approximations.

\medskip
The dual cone $\cgpsd^{n*}$ of the completely psd cone $\cgpsd^n$ has a useful  interpretation in terms of {\em trace positive non-commutative polynomials.}
A polynomial $p$ is {\em trace positive} if one gets a nonnegative value when evaluating $p$ at arbitrary  matrices $X_1,\ldots,X_n\in \mathcal S^d$ (for any $d\ge 1$) and taking the trace of the resulting matrix.
For a  matrix $M\in \mathcal S^n$, consider the following polynomial $p_M= \sum_{i,j=1}^n M_{ij} X_i^2 X_j^2$ in the non-commutative variables $X_1,\ldots,X_n$.
In Section~\ref{secdual}, we show that 
 $M$ belongs to the dual cone $\cgpsd^{n*}$ precisely when $p_M$ is trace positive.
%For a  matrix $M\in \mathcal S^n$, consider the following polynomial $p_M= \sum_{i,j=1}^n M_{ij} X_i^2 X_j^2$ in the non-commutative variables $X_1,\ldots,X_n$.
%In Section~\ref{secdual}, we show that 
% $M$ belongs to the dual cone $\cgpsd^{n*}$ precisely when $p_M$ is {\em trace positive}, which means that one gets a nonnegative value when evaluating $p_M$ at arbitrary  matrices $X_1,\ldots,X_n\in \mathcal S^d$ (for any $d\ge 1$) and taking the trace of the resulting matrix.
When restricting to scalar (commutative) variables we find the notion of copositive matrices and the fact that  $\cgpsd^{n*}$ is contained in the copositive cone $\cop^n$ (the dual of the completely positive cone $\cp^n$).

Trace positive polynomials have been  studied in the recent years, in particular in \cite{Burgdorf:2011,CKP12}. 
A sufficient condition for trace positivity of a polynomial $p$  is that $p$ belongs to the {\em tracial quadratic module} $\trMball$ (of the ball).
This means that
$p$ can be written as a sum of commutators $[g,h]=gh-hg$, Hermitian squares $gg^*$, and terms of the form $g(1-\sum_{i=1}^n X_i^2)g^*$,
where $g,h$ are non-commutative polynomials and $^*$ is the involution that  reverses the order of the variables in each monomial. 

It is shown in \cite{Klep:2008,Burgdorf:2013,Burgdorf:2011} that a celebrated conjecture of Connes in operator algebra is equivalent to showing that, for  any non-commutative polynomial $p$ which is trace positive,  the polynomial $p+\epsilon$ belongs to $\trMball $ for all $\epsilon>0$.
This motivates our definition of the convex set $\Knceps$, which consists  of all matrices $M$ for which the perturbed polynomial $p_M+\epsilon$ belongs to $\trMball$. 
Then, 
we have the inclusion $\bigcap_{\epsilon>0} \Knceps \subseteq \cgpsd^*$, 
 with equality if Connes' conjecture holds (Lemma~\ref{lembasic}).

Using these sets $\Knceps$, we can define the parameters $\Psi_\epsilon(G)$. Namely, $\Psi_\epsilon(G)$ is  obtained by considering   the (dual) optimization program over $\cgpsd^*$ which defines $\wtilde \chi_q(G)$ and replacing  in it the cone $\cgpsd^*$ by the convex set $\Knceps$ (see  Definition \ref{def:psi}).
Note that
 each parameter $\Psi_\epsilon(G)$ can be obtained as the limit of  a converging hierarchy of semidefinite programs obtained by introducing degree constraints on the terms in the tracial quadratic module $\trMball$ (see Section \ref{secPsi}).
We can show that $\Psi_\epsilon(G)$ relates to the theta number: $\Psi_\epsilon(G)\ge \vartheta^+(\overline G)$ (see Lemma~\ref{lem:Psi-epsilon}).
Unfortunately, as there is no apparent inclusion relationship between $\cgpsd^*$ and $\Knceps$, we do not know how $\Psi_\epsilon(G)$ compares to $\wtilde \chi_q(G)$ (and even less so to $\chi_q(G)$). However, if Connes' conjecture holds, then we have that $\wtilde \chi_q(G)\le \inf_{\epsilon >0} \Psi_\epsilon(G)$. 

Hence devising converging semidefinite approximations for the quantum chromatic number seems much harder than for its classical counterpart and 
 our results can be seen as a first step in this direction. 
 This difficulty should be put in the broader context of the general difficulty of approximating the quantum value of nonlocal games as mentioned earlier.

\medskip
Our main motivation for studying the cone $\cgpsd$ comes from its relevance to the quantum graph parameters.
Recent works have shown that the cone $\cgpsd$ can also be used to study the set of quantum bipartite correlations \cite{MR:2014} and the quantum value of any two-party game \cite{SV:2015}. 
Moreover, there is a further connection of this cone to  the widely studied notion  of factorizations of nonnegative matrices.
Given a nonnegative  $m\times n$ matrix $M$,  a {\em nonnegative factorization} (resp., a {\em psd factorization}) of $M$ consists  of nonnegative vectors 
$x_i,y_j\in \R_{+}^d$ (resp., psd matrices $x_i,y_j\in \mathcal S^d_+$) (for some $d\ge 1$) such that $M=(\langle x_i,y_j\rangle)_{i\in [m],j\in [n]}$.
Note that  asymmetric  factorizations are allowed, using $x_i$ for the rows and $y_j$ for the columns of $M$.
In this asymmetric setting, the question is not about the {\em existence} of a factorization (since such a factorization always exists in some dimension $d$), but about the {\em smallest possible dimension $d$} of such a factorization.
There is recently a surge of interest in these questions, motivated by the relevance of nonnegative factorizations (resp., psd factorizations) to  linear (resp., semidefinite) extended formulations of polytopes, see  e.g. \cite{Fiorini,GPT} and further references therein.

\subsection{Organization of the paper}

The paper is organized as follows.
In the rest of the introduction we present some notation and preliminaries about graphs and matrices used throughout.
 
 Section \ref{secparameters} introduces all graph parameters considered in the paper. Section \ref{secclassical} recalls   the classical parameters $\alpha(G)$, $\chi(G)$,  the theta numbers $\vartheta(G),$ $\vartheta'(G)$ and $\vartheta^+(G)$, and   two conic variants 
$\vartheta^\K(G)$ and $\Theta^\K(G)$ where $\K$ is a  cone nested between $\cp$ and $\dnn$.
Section~\ref{secquantum} introduces  the quantum graph parameters $\alpha_q(G),\qindep(G),\chi_q(G)$, $\qchrom(G)$ and, in Section \ref{secmotivation}, 
we briefly motivate the use of  these parameters for analyzing the impact of quantum entanglement in nonlocal games and in quantum information.

Section \ref{seccone} is devoted to the study of the new cone $\cgpsd$.   We discuss its basic properties (Section~\ref{secbasic}), the  links with $\cp$ and $\dnn$
(Section  \ref{seclink}),
  the  dual cone $\cgpsd^*$ and its link to trace positive polynomials (Section \ref{secdual}), and the convex sets $\Knceps$ (Section \ref{secapprsets}).

 Section \ref{secconicopt} shows how to reformulate   the quantum graph parameters using  linear  optimization  over affine sections of the cone $\cgpsd$.
First,  we reformulate  the quantum parameters as  checking  feasibility of a sequence of conic programs over sections of $\cgpsd$; this  is done  in Section \ref{secconealpha} for the quantum stability numbers and in Section \ref{secconechi} for the quantum chromatic numbers. 
We also show there how to recover the known bounds for the quantum graph parameters in terms of the theta number by replacing the cone $\cgpsd$ by the doubly nonnegative cone and establish new bounds given by the parameters $\vartheta^{\cgpsd}(G)$ and $\Theta^{\cgpsd}(G)$.  In Section \ref{secPsi}, we build a single `aggregated' optimization  program permitting to express the quantum parameter $\chi_q(G)$. 
The conic dual of this program is then used to define the parameter $\Psi_\epsilon(G)$, which is obtained by replacing the cone $\cgpsd^*$ by the convex set $\Knceps$  in this program.

Section \ref{secfinal} gives some concluding remarks and closes the paper.

\subsection{Notation and preliminaries}
\noindent
{\em Graphs.} Throughout,  all graphs are assumed to be finite, undirected and without loops. 
A graph $G$ has vertex set $V(G)$ and edge set $E(G)$. 
Given two vertices  $u,v\in V(G)$,  we write $u \simeq v$ if $u,v$ are adjacent or equal and we write $u\sim v$ when $u$ and $v$ are adjacent, in which case the corresponding edge is denoted as $\{u,v\}$ or simply as $uv$.
$\overline{G}$ is the complementary graph of $G$, with vertex set $V(G)$ and two distinct vertices are adjacent in $\overline{G}$  if and only if they are not adjacent in $G$. 

A stable set of $G$  is a subset of $V(G)$ where any two vertices are not adjacent. The {\em stability number} $\alpha(G)$ is the maximum cardinality of a stable set in $G$. A clique of $G$ is a set of vertices that are pairwise adjacent and $\omega(G)$ is the maximum cardinality of a clique; clearly, $\omega(G)=\alpha(\overline{G})$.
A {\em proper coloring} of $G$ is a coloring of the vertices of $G$ in such a way that adjacent vertices receive distinct colors. The {\em chromatic number} $\chi(G)$ is the minimum number of colors needed for a proper coloring.
Equivalently, $\chi(G)$ is the smallest number of stable sets needed to cover all vertices of $G$. The {\em fractional chromatic number} $\chi_f(G)$ is the fractional analogue, defined as the smallest value of $\sum_{h=1}^k \lambda_h$ for which there exists stable sets $S_1,\ldots,S_k$ of $G$ and nonnegative scalars $\lambda_1,\ldots,\lambda_k$ such that $\sum_{h: v\in S_h}\lambda_h =1$ for all $v\in V(G)$.
Clearly, $\omega(G)\le \chi_f(G)\le \chi(G)$.

%For $t \in \N$, we set $[t]=\{1,\ldots,t\}$ and $K_t$ denotes the complete graph on $[t]$.
For $t \in \N$, we set $[t]=\{1,\ldots,t\}$, $K_t$ denotes the complete graph on $[t]$ and $C_{n}$ the $n$-cycle.
The graph $G \square K_t$ is the {\em Cartesian product} of  $G$ and $K_t$. Its vertex set is $V(G)\times [t]$ and two vertices $(u,i)$ and $(v,j)$ are adjacent in $G\Box K_t$  if ($u=v$ and $i \ne j$) or if ($u \sim v$ and $i=j$).

\medskip\noindent
{\em Cones and matrices.} 
Throughout, $\R^n_+$ denotes the set of (entrywise) nonnegative vectors, $e_1,\ldots,e_n$ denote the standard unit vectors in $\R^n$, and $e$ denotes the all-ones vector. $\R^n$ is equipped with the standard inner product: $\langle x,y\rangle =x^Ty=\sum_{i=1}^n x_iy_i$  and the corresponding norm
$\|x\|=\sqrt{\langle x,x\rangle}$.

We denote by $\mathcal S^n$ the set of $n \times n$ real symmetric matrices, which is equipped  with the standard  trace inner product: $\langle X,Y\rangle =\Tr(XY)=\sum_{i,j=1}^n X_{ij}Y_{ij}$ and the corresponding Frobenius norm
$\|X\| = \sqrt{\langle X,X \rangle}$.
For $X\in \mathcal S^n$, $X$ is positive semidefinite (also written as   $X \succeq 0$) if all its eigenvalues are nonnegative. 
We let $\cpsd^n$ denote the set of positive semidefinite matrices in $\mathcal S^n$ and $\dnn^n$, the {\em doubly nonnegative cone}, is the set of positive semidefinite matrices in $\mathcal S^n$ with nonnegative entries. 
As it is well known, 
$X\succeq 0$ if and only if there exist vectors $x_1,\ldots,x_n\in \R^d$ (for some $d\ge 1$) such that $X_{ij}=\langle x_i,x_j\rangle$ for all $i,j\in [n]$, in which case  we say that $x_1,\ldots,x_n$ form a {\em Gram representation} of $X$ and we call $X$ the {\em Gram matrix} of $x_1,\ldots,x_n$. 
Furthermore, $X\in\mathcal S^n$ is said to be {\em completely positive} if $X$ is the Gram matrix of a set of {\em nonnegative} vectors $x_1,\ldots,x_n\in \R^d_+$ (for some $d\ge 1$). We let $\cp^n$ denote the set of completely positive matrices. 
The sets
$\mathcal S^n_+$, $\dnn^n$ and $\cp^n$ are all convex cones.
%We may sometimes omit the superscript and use the notation $\mathcal S_+$, $\dnn$ and $\cp$.

For a pair of matrices $X,Y$,  $X \oplus Y = \begin{pmatrix}
X & 0 \\
0 & Y
\end{pmatrix}$ denotes their direct sum, 
$X \circ Y$ denotes the entrywise product, where the $ij$-th entry of $X \circ Y$ is equal to  $ X_{ij} Y_{ij}$,
 and  $X \otimes Y$ their Kronecker product, defined as 
 the block matrix $X \otimes Y=\begin{pmatrix}
X_{11} Y & \ldots & X_{1n} Y \\
\vdots & \ddots & \vdots \\
X_{m1} Y & \ldots & X_{mn} Y
\end{pmatrix}$
if $X$ is $m\times n$.
If $X,Y \succeq 0$, then also $X \oplus Y$, $X \circ Y$ and $X \otimes Y$ are positive semidefinite.

We will also use the following elementary facts.
First, $nI-J\succeq 0$, where $I,J$ are the identity and all-ones matrix in $\mathcal S^n$. If $X,Y \succeq 0$, then $\langle X,Y \rangle = 0$ if and only if $XY = 0$.
Moreover, for  $X\in \mathcal S^n$ of the form
$X=\left(\begin{matrix} \alpha & b^T\cr b & A\end{matrix}\right),$ where  
$b\in \R^{n-1}, \ A\in \mathcal S^{n-1} \text{ and } \alpha >0,$
\begin{equation}\label{eqSchur}
X\succeq 0\  \Longleftrightarrow \ %\alpha>0 \ \text{ and } \ 
A- %{\frac{1}{\alpha}} 
bb^T/\alpha \succeq 0.
\end{equation}
The matrix $A-%{\frac{1}{\alpha}}
 bb^T/\alpha$ is called the Schur complement of $A$ in $X$ w.r.t. the entry $\alpha$.

Given a cone $\mathcal K \subseteq \mathcal S^n$, its dual cone is the cone
$\mathcal 
K^* = \{ X \in S^n \st \langle X,Y\rangle  \geq 0 \quad \forall \, Y \in \mathcal K \}.$
Recall that the cone $\mathcal S^n_+$ is self-dual, i.e., $\mathcal S_+^{n*}=\mathcal S^n_+$.  
Given $C,A_j\in \mathcal S^n$ and $b_j\in\R$ for $j\in [m]$, consider the following pair of primal and dual conic programs over a nice cone $\K$ (i.e., $\K$ is closed, convex,  pointed and full-dimensional):
\begin{eqnarray}
p^*=\sup \ \langle C,X\rangle \ \text{ s.t. } \langle A_j,X\rangle =b_j\ \forall j\in [m],\ X\in \K, \label{eqP}\\
d^*=\inf \ \sum_{j=1}^m b_jy_j \ \text{ s.t. } Z=\sum_{j=1}^m y_jA_j-C\in \K^*.\label{eqD}
\end{eqnarray}
Weak duality holds: $p^*\le d^*$. Moreover, assume that $d^*>-\infty$ and  (\ref{eqD}) is strictly feasible (i.e., has a feasible solution $y,Z$ where $Z$ lies in the interior of $\K^*$), then strong duality holds: $p^*=d^*$ and (\ref{eqP}) attains its supremum.

\section{Classical and quantum graph parameters}\label{secparameters}

\subsection{Classical graph parameters}\label{secclassical}
We group here several preliminary results about classical graph parameters that we will need in the paper. In what follows $G$ is a graph on $n$ vertices.
We begin with the following  result of Chv\'atal~\cite{Chvatal:1973} which shows how to relate the chromatic number of $G$ to the stability number of the Cartesian product $G \Box K_t$.

\begin{theorem}\label{theoChvatal}\cite{Chvatal:1973}
For any graph $G$ and any integer $t\ge 1$, 
$\chi(G) \leq t \iff \alpha(G \square K_t) = |V(G)|.$
Hence, $\chi(G)$ is equal to the smallest integer $t$ for which  $\alpha(G \square K_t) = |V(G)|$ holds.
\end{theorem}

\noindent
Next we recall  the following  reformulation for  the stability number $\alpha(G)$ as an optimization problem over the completely positive cone, which was proved by de Klerk and Pasechnik~\cite{deKlerk:2002}. 

\begin{theorem}~\cite{deKlerk:2002}\label{theodeK}
For any graph $G$, its stability number $\alpha(G)$ is equal to the optimum value of the following  program:
\begin{equation}\label{def:stab}
\begin{aligned}
 \max \; \langle J,X \rangle \quad \text{s.t.} \quad & X \in \cp^{n}, \quad
 \Tr(X)= 1,
& X_{uv} = 0 \quad \forall \{u,v\} \in E(G).
\end{aligned}
\end{equation}
\end{theorem}

\noindent
In the same vein, Dukanovic and Rendl \cite{Dukanovic:2010} gave the following reformulation for the fractional chromatic number $\chi_f(G)$.

\begin{theorem}\cite{Dukanovic:2010} \label{theoDR}
For any graph $G$, its fractional chromatic number $\chi_f(G)$ is equal to the optimum value of the following program:
\begin{equation*}
\begin{aligned}
\min \; t \ \text{ s.t. }  X\in \cp^{n}, \quad X-J\succeq 0, \quad X_{uu}=t \quad \forall u\in V(G),
 \quad X_{uv} = 0 \quad \forall \{u,v\} \in E(G).
\end{aligned}
\end{equation*}
\end{theorem}

\noindent
A well known  bound for both the stability and the (fractional) chromatic numbers is provided by the celebrated  theta number $\vartheta$ of  Lov\'asz~\cite{Lovasz:1979}, who showed 
the following  `sandwich' inequalities:
\begin{equation}\label{eqsandwich}
\alpha(G) \leq \vartheta(G) \leq \chi_f(\overline{G})\le \chi(\overline G).
\end{equation}
Between the many equivalent formulations of the  theta number, the following will be appropriate for our setting:
\begin{equation}\label{opt:theta-stand}
\begin{split}
\vartheta(G) =  \max \quad & \langle J,X \rangle \\
 \text{s.t.} \quad  & X \succeq 0\\
& \Tr(X) = 1\\
&X_{uv} = 0 \;\; \forall \, \{u,v\} \in E(G);\\
\end{split}
\begin{split}
= \min \quad & t \\
\text{s.t.} \quad & Z \succeq 0, \; Z-J \succeq 0\\ % \; t \geq 2, \\
& Z_{uu} = t \;\; \forall \, u \in V(G)\\
&Z_{uv} = 0 \;\; \forall \, \{u,v\} \in E(\overline G).
\end{split}
\end{equation}

\noindent
In view of Theorem \ref{theodeK}, if in the above maximization program defining $\vartheta(G)$ we replace the condition $X\succeq 0$ by the condition $X\in \cp$, then the optimal value is equal to $\alpha(G)$. 
Similarly, in view of Theorem \ref{theoDR},  
$\chi_f(\overline G)$ is the optimal value of the above minimization program defining $\vartheta(G)$ when, instead of requiring that $Z \succeq 0$, we impose the condition $Z \in \cp$.

Several strengthenings of $\vartheta(G)$ toward $\alpha(G)$ and $\chi(G)$ have been proposed, in particular, the following 
parameters $\vartheta'(G)$ introduced independently by 
Schrijver~\cite{Schrijver:1979} and  McEliece et al.~\cite{McEliece:1978} and $\vartheta^+(G)$ introduced by  Szegedy~\cite{Szegedy:1994}:
\begin{equation}\label{opt:theta-str}
\begin{split}
\vartheta'(G) =  \max \quad & \langle J,X \rangle \\
 \text{s.t.} \quad  & X\in \dnn^n \\ %X \succeq 0,\; X \geq 0\\
& \Tr(X) = 1\\
&X_{uv} = 0 \;\; \forall \, \{u,v\} \in E(G);\\
\end{split}
\begin{split}
\vartheta^+(G) =  \min \;\; & t \\
\text{s.t.} \;\; &  Z\in \dnn^n, \, Z-J\succeq 0 \\ 
& Z_{uu} = t \;\; \forall \, u \in V(G)\\
&Z_{uv} =0  \;\; \forall \, \{u,v\} \in E(\overline G).\\
\end{split}
\end{equation}

\noindent
The following inequalities hold, which refine (\ref{eqsandwich}):
\begin{equation}\label{eqsandwichp}
\alpha(G) \le \vartheta'(G)\le \vartheta (G) \le \vartheta^+(G) \le  \chi_f(\overline G)\le \chi(\overline G).
\end{equation}

\medskip
Following \cite{Dukanovic:2010}, we now introduce the following conic programs (\ref{eqthetas-K}), which are  obtained by replacing in the above programs (\ref{opt:theta-stand})  the positive semidefinite cone by a general convex cone $\K$ nested between   the cones $\cp$ and $\dnn$.
Namely, given a graph $G$, we consider the following parameters $\vartheta^\K(G)$ and $\Theta^\K(G)$,
  which we will use later in Sections \ref{seclink}, \ref{secconealpha} and \ref{secconechi}:

\begin{equation}\label{eqthetas-K}
\begin{split}
\vartheta^{\mathcal K}(G) =  \sup \ & \langle J,X \rangle \\
 \text{s.t.} \  & X\in \mathcal{K}^{n}\\
& \Tr(X) = 1\\
&X_{uv} = 0 \;\; \forall \ \{u,v\} \in E(G);\\
\end{split}
\begin{split}
\Theta^{\mathcal K}(G) =  \inf \;\; & t \\
\text{s.t.} \;\; &   Z\in \mathcal{K}^n,  \, Z-J\succeq 0 \\ 
& Z_{uu} = t \;\; \forall \, u \in V(G)\\
&Z_{uv} =0  \;\; \forall \, \{u,v\} \in E(G).\\
\end{split}
\end{equation}

If in the relations from (\ref{eqthetas-K}), we set $\K=\dnn$ or $\K=\cp$ then, using the above definitions and   Theorems \ref{theodeK} and \ref{theoDR},  we find respectively the definitions  of $\vartheta'(G)$, $\alpha(G)$, $\vartheta^+(\overline G)$, $\chi_f(G)$. That is,
\begin{equation}\label{eqlinktheta}
\vartheta^{\dnn}(G)=\vartheta'(G),\ \vartheta^\cp(G)=\alpha(G), \
\Theta^\dnn(G)=\vartheta^+(\overline G), \ \Theta^\cp(G)=
\chi_f(G).
\end{equation}

\begin{remark}\label{remK}
We observe a `monotonicity' property for the  program defining $\vartheta^{\mathcal K}(G)$ in (\ref{eqthetas-K}), that will be useful later.
Set $n=|V(G)|$ and consider scalars   $1\le t< T$.
Assume that a matrix $X$ is feasible for the program defining $\vartheta^\K(G)$ with value $\langle J,X\rangle =T$.
 Then the matrix  $X'={t-1\over T-1}X+{T-t\over n(T-1)}I$ is again feasible for   $ \vartheta^\K(G)$ and it has value   $\langle J,X'\rangle =t$. 
\end{remark}

\subsection{Quantum graph parameters}\label{secquantum}

We now introduce  two `quantum' variants  $\alpha_q(G)$ and $\qindep(G)$ of the stability number and two `quantum' variants $\chi_q(G)$ and $\qchrom(G)$ of the chromatic number, which have been considered  in the literature.
Here we will give mathematical definitions.
 Motivation for these parameters will be given in Section \ref{secmotivation} below.

\begin{definition}[Game entanglement-assisted stability number~\cite{Roberson:2012}]\label{def:gqindep}
For a graph $G$,   $\alpha_q(G)$ is the maximum integer $t\in\N$ for which  there exist positive semidefinite matrices $\rho, \rho_i^u\in \mathcal S^d_+$ for $i\in[t],\, u\in V(G)$ (for some $d\ge 1$) satisfying the following conditions:
\beqr
\langle \rho,\rho\rangle  &=& 1, \label{eq:qalpha-1}\\
\sum_{u\in V(G)} \rho^u_i &=& \rho \quad \forall i\in[t] \label{eq:qalpha-sum}, \\
\langle\rho^u_i,\rho^v_j \rangle &=& 0 \quad \forall i\ne j \in [t],\, \forall u\simeq v \in V(G), \label{eq:qalpha-ort1}\\
\langle\rho^u_i,\rho^v_i \rangle &=& 0 \quad \forall i\in [t], \,  \forall u\ne v \in V(G).\label{eq:qalpha-ort2}
\eeqr
\end{definition}

\begin{definition}[Communication entanglement-assisted stability number~\cite{Cubitt:2010}]\label{def:qindep}
For a graph $G$,   $\qindep(G)$ is the maximum $t\in\N$ for which  there exist positive semidefinite matrices $\rho, \rho_i^u\in \mathcal S^d_+$ for $i\in[t],\, u\in V(G)$ (for some $d\ge 1$) satisfying the  conditions
(\ref{eq:qalpha-1}), (\ref{eq:qalpha-sum}) and (\ref{eq:qalpha-ort1}).
\end{definition}

\begin{definition}[Game entanglement-assisted chromatic number~\cite{Cameron:2007}]\label{def:gqchrom}
For a graph $G$,  $\chi_q(G)$ is the minimum $t\in\N$ for which  there exist  positive semidefinite matrices $\rho, \rho^i_u\in \mathcal S^d_+$ for $i \in [t]$, $u \in V(G)$   (for some $d\ge 1$) satisfying the following conditions:
\beqr
\langle \rho,\rho\rangle &=& 1, \label{eq:qchrom-1}\\
\sum_{i\in [t]} \rho^i_u &=& \rho \quad \forall u\in V(G), \label{eq:qchrom-sum} \\
\langle \rho_u^i, \rho_v^i \rangle &=& 0 \quad \forall i\in [t], \forall \{u, v\} \in E(G),\label{eq:qchrom-ort1}\\
\langle \rho_u^i, \rho_u^j \rangle &=& 0 \quad \forall i \ne j \in [t], \forall u \in V(G).\label{eq:qchrom-ort2}
\eeqr
\end{definition}

\begin{definition}[Communication entanglement-assisted chromatic number~\cite{Briet:2013}] \label{def:qchrom}
For a graph $G$,  $\qchrom(G)$ is the minimum~$t\in\N$ for which there exist  positive semidefinite matrices $\rho, \rho_u^i\in \mathcal S^d_+$ for $ i\in[t]$, $u\in V(G)$ (for some $d\ge 1$) satisfying the conditions  (\ref{eq:qchrom-1}), (\ref{eq:qchrom-sum}) and (\ref{eq:qchrom-ort1}).
\end{definition}

\noindent
The parameters $\alpha_q(G)$ and $\chi_q(G)$ can, respectively, be equivalently obtained from the definitions of $\qindep(G)$ and $\qchrom(G)$ if we require $\rho$ to be the identity matrix (instead of $\langle \rho,\rho\rangle  = 1$) and  the other positive semidefinite matrices to be orthogonal projectors, i.e., to satisfy $\rho^2 = \rho$ (see \cite{Roberson:2012} and \cite{Cameron:2007}). 
Moreover, if in the programs of Definitions~\ref{def:gqindep} and~\ref{def:qindep}  we require the variables $\rho^{i}_{u}$ to be $0 / 1$ valued, then we obtain the classical parameter $\alpha(G)$. 
Similarly, if in Definitions~\ref{def:gqchrom} and~\ref{def:qchrom} we restrict the variables $\rho^{i}_{u}$ to be $0 / 1$ valued, then we obtain program (\ref{eq0}) and thus the classical parameter $\chi(G)$.
Therefore, the following inequalities hold:
$$\alpha(G)\le \alpha_q(G)\le \qindep(G) \text{ and } \qchrom(G)\le \chi_q(G)\le \chi(G).$$

\noindent
It is not known whether the parameters $\alpha_q(G)$ and $\qindep(G)$ are in general equal or not. The same question holds for $\chi_{q}(G)$ and $\qchrom(G)$.
Recently, several  bounds for the quantum parameters have been established in terms of the theta number.
Namely, \cite{Beigi:2010,Duan:2013} show  that
$\qindep(G)\le \vartheta(G)$, \cite{Cubitt:2013} shows the tighter bound $\qindep(G)\le \vartheta'(G)$, and 
\cite{Briet:2013} shows that $\qchrom(G)\ge \vartheta^{+}(\overline{G})$.  Summarizing, the following sandwich inequalities hold:
\begin{equation}\label{eqsandwichq}
\alpha(G)\le \alpha_q(G)\le \qindep(G)\le \vartheta'(G)  \text{ and } \vartheta^{+}(\overline{G})\le \qchrom(G)\le \chi_q(G)\le \chi(G).
\end{equation}
Using our approach of reformulating the quantum parameters as optimization problems over the cone $\cgpsd$, we will recover these bounds (see Section \ref{secconicopt}, in particular, Corollaries~\ref{corsandqa} and \ref{corsandqc}).

\subsection{Motivation}\label{secmotivation}

The quantum graph parameters that we have just defined arise in the general context of the study of entanglement, one of  the most important features of quantum mechanics.
In particular, the parameters $\alpha_q(G)$ and $\chi_q(G)$ are defined in terms of {\em nonlocal games}, which are mathematical abstractions of a physical experiment introduced by \cite{Cleve:2004}.
In a nonlocal game, two (or more) cooperating players determine a common strategy to answer questions posed by a referee.
The questions are drawn from a finite set and the referee sends a question to each of  the players.
The players, without communicating, must each respond to their question and the referee upon collecting all the answers determines according to the rules of the game whether the players win or lose. 
We can now study properties of quantum mechanics, by asking the following question: Does entanglement between the players allow for a better strategy than the best classical one?
Surprisingly, players that share entanglement can (for some games) produce answers correlated in a way that would be impossible in a classical world. 
%(For an introduction to the topic we recommend the book \cite{Nielsen:2000}.)
(For an introduction to the topic we recommend the survey \cite{Cleve:2004}.)

For a fixed integer  $t \in \mathbb{N}$,
consider a game where two players want to convince a referee that they can color a graph $G$ with at most $t$ colors.
The players each receive a vertex from the referee and they answer by returning a color from $[t]$. They win the game if they answer the same color upon receiving the same vertex and different colors if the vertices are adjacent.
When the players use classical strategies, they can always win if and only if there exists a $t$-coloring of the graph.
In other words, the chromatic number $\chi(G)$ is the minimum $t$ for which the players can always win the game using classical strategies.
In the entanglement-assisted setting, $\chi_q(G)$ is the smallest number of colors that the players must use in order to always win the game (see \cite{Cameron:2007} for details). 
This parameter has recently received a notable amount of attention (see among others \cite{Avis:2006,Cameron:2007,Fukawa:2011,Mancinska:2012, Roberson:2012,Ji:2013,Paulsen:2013}).

Analogously to $\chi_q(G)$, $\alpha_q(G)$ is the maximum integer $t$ for which two players sharing an entangled state can convince a referee that the graph $G$ has a stable set of cardinality $t$. For a detailed description of the game we refer to \cite{Roberson:2012}. 

From the above description, it is not directly clear that Definitions~\ref{def:gqindep} and \ref{def:gqchrom} are the proper mathematical formulations for the parameters $\alpha_{q}(G)$ and $\chi_{q}(G)$ and this indeed requires a proof (which can be implicitly found in \cite{Roberson:2012} for $\alpha_{q}(G)$ and in \cite{Cameron:2007} for $\chi_{q}(G)$).
We also refer the interested reader to \cite[Sections 6.2, 6.4 and 6.5]{Roberson:PhD}, where a direct mathematical proof which does not require quantum information background is given.
Furthermore note that in the above mentioned references, the matrices $\rho_{i}^{u}$ (or respectively $\rho_{u}^{i}$) are required to be projectors and $\rho$ to be the identity matrix.  By replacing each matrix $\rho_{i}^{u}$ by the projection onto its image,  it can be seen that we can equivalently consider Definitions~\ref{def:gqindep} and \ref{def:gqchrom}.

Another setting where the properties of entanglement can be studied is zero-error information theory.
Here two parties want to perform a communication task (e.g., communicating through a noisy channel) both exactly and efficiently.
These problems have led  to the development of a new line of research in combinatorics 
(see \cite{Korner:1998} for a survey and references therein).
Recently Cubitt et al.~\cite{Cubitt:2010} started studying whether sharing entanglement between the two parties can improve the communication. A number of positive results, where entanglement does improve the communication, have been obtained \cite{Cubitt:2010,Leung:2012,Briet:2013}.
Without getting into details, the parameters $\qindep(G)$ and $\qchrom(G)$ arise in this entanglement-assisted information theory setting.
For the full description of the problem and its mathematical formulation we refer to \cite[Theorem 1]{Cubitt:2010} and \cite[Definition I.5]{Briet:2013} for $\qindep(G)$ and $\qchrom(G)$, respectively. 

\medskip
We now briefly summarize some known properties of these parameters.
One of the most interesting questions is to find and characterize graphs for which there is a separation between a quantum parameter and its classical counterpart.
Clearly there is no such separation when $G$ is a perfect graph since then  the inequalities (\ref{eqsandwichp}) and (\ref{eqsandwichq}) imply $\alpha(G) = \alpha_q(G) = \qindep(G) = {\vartheta(G)}$ and ${\vartheta(\overline G)} = \qchrom(G) = \chi_q(G) = \chi(G)$. 
Indeed, if a graph $G$ is perfect then both $\alpha(G) = \chi(\overline G)$ and $\alpha(\overline G) = \chi(G)$ hold. Therefore, equality holds throughout in (\ref{eqsandwichp}) and in (\ref{eqsandwichq}).

A few separation results are known. For instance, there exists a graph for which $\qchrom(G) = \chi_q(G) = 3$ but $\chi(G) = 4$ \cite{Fukawa:2011}, and a family of graphs exhibiting an exponential separation between $\chi_q$ and $\chi$ 
\cite{Avis:2006}
 (and therefore also between $\qchrom$ and $\chi$).
This family is composed of the {\em orthogonality graphs}  $\Omega_n$ (where $n$ is a multiple of 4) whose vertices are all  the vectors in $\{\pm 1\}^n$  and two vertices are adjacent if the vectors are orthogonal. 
%In  \cite{Mancinska:2012,Roberson:2012} 
These graphs have been used to construct graphs that exhibit an exponential separation between $\qindep$ and $\alpha$ in \cite{Mancinska:2012} and between $\alpha_q$ and $\alpha$ in \cite{Roberson:2012}.

While for the classical parameters the inequality  $\chi(G) \alpha(G)\ge |V(G)|$ holds for any graph $G$, interestingly this is not true for the quantum counterparts.
As noticed in \cite{Roberson:2012}, if $n$ is a multiple of 4 but not a power of 2, then $\chi_q(\Omega_n)\alpha_q(\Omega_n) < |V(\Omega_n)|$ and the exact same reasoning implies that $\qchrom(\Omega_n)\qindep(\Omega_n) < |V(\Omega_n)|$.
Finally, the chromatic and stability number are NP-hard quantities and  recently Ji \cite{Ji:2013} proved that deciding whether $\chi_q(G) \le 3$ is an NP-hard problem. 
Roberson and Man\v{c}inska~\cite[Lemma 4.5]{Roberson:2012} showed the analogue of Theorem~\ref{theoChvatal} for the quantum parameters $\alpha_{q}$ and $\chi_{q}$, namely that $\chi_{q}(G) \le t$ if and only if $\alpha_{q}(G \square K_{t}) = |V(G)|$, and thus computing the parameter $\alpha_{q}(G)$ is NP-hard as well.
%Combining this result with \cite[Lemma 4.5]{Roberson:2012}, we derive that also the parameter $\alpha_{q}(G)$ is NP-hard.

\section{The completely positive semidefinite cone}\label{seccone}

In this section we introduce the completely positive semidefinite cone $\cgpsd$, we establish some of its basic properties and  its relation with  the  completely positive cone and with the doubly nonnegative cone. We also  
  investigate its  dual cone $\cgpsd^{*}$ and we introduce the   convex sets $\Knceps$ aiming to approximate it.

\subsection{Basic properties}\label{secbasic}

Recall that for any positive semidefinite matrix $A$ there exists a set of vectors $x_1,\ldots,x_n \in \R^d$ that form its Gram representation, i.e., $A=(\langle x_i,x_j\rangle)_{i,j=1}^n$.
We now consider Gram representations by {\em positive semidefinite matrices}.

\begin{definition}
A matrix $A\in \mathcal S^n$ is said to be {\em completely positive semidefinite} (completely psd, for short) if  there exist matrices $X_1,\ldots, X_n\in \mathcal S^d_+$ (for some $d\ge 1$) such that 
$A=(\langle X_i,X_j\rangle )_{i,j=1}^n$. Then we also say that $X_1,\ldots,X_n$ form a  Gram representation of $A$. We let 
$\cgpsd^n$ denote the set of all $n \times n$ completely positive  semidefinite matrices.
\end{definition}

\begin{lemma}\label{lemconvex}
$\cgpsd^n$ is a convex cone.
\end{lemma}

\begin{proof}
Let $A,B\in \cgpsd^n$ and assume that  $X_1,\ldots,X_n\in\mathcal S^d_+$ and $Y_1,\ldots,Y_n\in \mathcal S^{k}_+$ form a Gram representation of $A$ and $B$, respectively. 
Then, the matrices $X_1\oplus Y_1,\ldots, X_n\oplus Y_n$ are psd and form a Gram representation of $X+Y$, thus showing that $X+Y\in \cgpsd^n$.

Moreover, let $\lambda \ge 0$ and consider the psd matrices $\sqrt \lambda X_1, \ldots, \sqrt \lambda X_n$. These form a Gram representation of $\lambda A$, thus showing that $\lambda A\in \cgpsd^n$.
Hence $\cgpsd^n$ is a convex cone.
\end{proof}

As is well known, both  $\mathcal S^n_+$ and $\cp^n$  are closed sets. This is easy for $\mathcal S^n_+$, since it is a self-dual cone.
For $\cp^{n}$, this can be proven using the fact that its extreme rays %of the completely positive cone 
are the rank 1 matrices $y y^{T}$ where $y \in \R^{n}_{+}$.
Therefore, any matrix in $\cp^n$ can be written as  $\sum_{i=1}^N y_iy_i^T$, where $y_1,\ldots,y_N\in\R^n_+$ and  $N \leq {n+1 \choose 2}$ (using Carath\'eodory's theorem), and thus closedness follows using a compactness argument (see, \eg,~\cite[Theorem 2.2]{Berman:2003} for the full proof).
Having an explicit description of the extreme rays of the $\cp$ cone is a key ingredient in many proofs concerning  $\cp$.
One of the difficulties in proving properties of the completely psd cone lies in the fact that we do not know an alternative  description for it; for example, we do not  know its extreme rays. 
In particular, one of the most interesting properties that we do not know is whether 
%Interestingly, we do not  know whether 
the cone $\cgpsd^n$ is  closed.
%For the cone $\cp^n$, this can be seen as follows: any matrix in $\cp^n$ can be written as a sum of rank 1 matrices $\sum_{i=1}^N y_iy_i^T$, where $y_1,\ldots,y_N\in\R^n_+$ and where $N \leq {n+1 \choose 2}$ (using Carath\'eodory's theorem) and thus closedness follows using a compactness argument \colred{(see \eg~\cite[Theorem 2.2]{Berman:2003})}.
%Interestingly, we do not  know whether the cone $\cgpsd^n$ is  closed as well.\footnote{\colblue{add that we don't have analogue representation of the $\cgpsd^n$ cone}}

As we now see, deciding whether the cone $\cgpsd$  is closed is  related to the  following question: Does the existence of a Gram representation by {\em infinite} positive semidefinite matrices imply the existence of another Gram representation by positive semidefinite matrices of {\em finite} size?
More precisely, let $\mathcal S^{\N}$ denote the set of all infinite symmetric matrices $X=(X_{ij})_{i,j\ge 1}$ 
with finite norm: $\sum_{i,j\ge 1}X_{ij}^2 <\infty$. Thus $\mathcal S^\N$ is a Hilbert space, equipped with the inner product $\langle X,Y\rangle =\sum_{i,j\ge 1}X_{ij}Y_{ij}$.
Call a matrix $X\in  \mathcal S^\N$ psd (again denoted as $X\succeq 0$) when all its finite principal submatrices are psd, i.e., $X[I]\in \mathcal S^{|I|}_+$ for all finite subsets $I\subseteq \N$, and let $\mathcal S^\N_+$ denote the set of all psd matrices in $\mathcal S^{\N}$.
Finally,  let $\mathcal{CS}_{\infty +}^n$ denote the set of matrices $A\in \mathcal S^n$ having a Gram representation by elements of $\mathcal S^\N_+$.
As for $\cgpsd^n$, one can verify that $\mathcal{CS}_{\infty +}^n$ is a convex cone. Moreover we can show the following relationships  between these two cones.

\begin{theorem}\label{theoclosed}
For any $n \in \N$,  $\cgpsd^n \subseteq \mathcal{CS}_{\infty +}^n \subseteq \mathrm{cl}(\mathcal {CS}^n_{\infty +}) = \mathrm{cl}(\cgpsd^n)$ holds.
\end{theorem}

\begin{proof}
The inclusion $\cgpsd^n \subseteq \mathcal{CS}_{\infty +}^n$ is clear, since any matrix $X\in \mathcal S^d_+$ can be viewed as an element of $\mathcal S^\N_+$ by adding zero entries.

We now prove the inclusion: $\mathcal{CS}_{\infty +}^n \subseteq  \mathrm{cl}(\cgpsd^n)$.
For this, let $A \in \mathcal{CS}_{\infty +}^n$ and $X_1,\ldots,X_n\in \mathcal S^\N_+$  be a Gram representation of $A$, i.e., 
$A_{ij} = \langle X_i,X_j \rangle$ for $i,j \in [n]$.
For any $\ell \in \N$ and $i \in [n]$, let $X_i^\ell=X_i[\{1,\ldots,\ell\}]$ be the $\ell \times \ell$ upper left principal submatrix of $X_i$ and let
$\widetilde{X}_i^\ell\in \mathcal S^\N$ be the infinite matrix obtained by adding zero entries to $X_i^\ell$. Thus, $X^\ell_i\in \mathcal S^\ell_+$ and $\widetilde X^\ell_i \in \mathcal S^\N_+$.
Now, let $A^\ell$ denote the Gram matrix of $X^\ell_1,\ldots,X^\ell_n$, so that  $A^\ell \in \cgpsd^n$. 
We  show that the sequence $(A^\ell)_{\ell\ge 1}$ converges to $A$ as $\ell$ tends to $\infty$, which shows that $A\in \mathrm{cl}(\cgpsd^n)$. 
 Indeed, for any $i,j \in [n]$ and $\ell \in \N$, we have:
\beqrn
|A_{ij} - A_{ij}^\ell| 
&=& | \langle X_i,X_j \rangle - \langle X_i^\ell,X_j^\ell \rangle |\\
&\le & | \langle X_i- \widetilde{X}_i^\ell, X_j \rangle | + | \langle \widetilde{X}_i^\ell,X_j-\widetilde{X}_j^\ell \rangle | \\
&\leq & \| X_i - \widetilde{X}_i^\ell \| \| X_j \| + \| \widetilde{X}_i^\ell \| \| X_j - \widetilde{X}_j^\ell \|,
\eeqrn
using the Cauchy-Schwarz inequality in the last step.
Clearly, $\|\widetilde{X}_i^\ell \| \leq \|X_i\|=\sqrt{ A_{ii}}$ for all $\ell\in\N$ and $i\in [n]$.
Hence $\lim_{\ell \to \infty} | A_{ij} - A_{ij}^\ell| = 0$ for all $i,j\in [n]$, concluding the proof.

Taking the closure in the inclusions: $\cgpsd^n \subseteq \mathcal{CS}_{\infty +}^n \subseteq \mathrm{cl}(\cgpsd^n)$, we conclude that $ \mathrm{cl}(\mathcal {CS}^n_{\infty +}) = \mathrm{cl}(\cgpsd^n)$ holds.
\end{proof}

The recent work \cite{BLP} studies the closure of the cone $\cgpsd$. In particular it gives an interpretation of $\cl (\cgpsd^{n})$ in terms of Gram representations by positive elements in some  von Neumann algebra. We will use this in the next section to show a separation between the closure of $\cgpsd^5$ and $\dnn^5$. We also refer to Section \ref{secfinal} for some further comments.

\subsection{Links to  completely positive and doubly nonnegative matrices}\label{seclink}

The following relationships follow from the definitions:
\begin{equation}\label{eqinclusion}
\cp^n\subseteq \cgpsd^n\subseteq \mathcal S^n_+\cap \R^{n\times n}_+=:\dnn^n.
\end{equation}
As the cone $\cp^n$ is full-dimensional, the same holds for the cone $\cgpsd^n$. 
% is full-dimensional and pointed. 
That every completely positive semidefinite matrix is entrywise nonnegative follows from the fact that $\langle X,Y\rangle \ge 0$ for all $X,Y\in \mathcal S^n_+$. 
Taking duals in (\ref{eqinclusion}) we get the corresponding inclusions:
\begin{equation}\label{eqinclusiond}
\dnn^{n*} \subseteq \cgpsd^{n*}\subseteq \cp^{n*}.
\end{equation}
The dual of $\cp^n$ is the {\em copositive cone}, which consists of all matrices $M\in \mathcal S^n$ that are {\em copositive}, i.e., satisfy $x^TMx\ge 0$ for all $x\in \R^n_+$. The dual of $\dnn^n$ is the cone $\mathcal S^n_+ + (\mathcal S^n \cap \R^{n\times n}_+)$.
We will investigate the dual of $\cgpsd^n$ in detail in the next section.

\medskip
We now present some results regarding the inclusions in (\ref{eqinclusion})
and  (\ref{eqinclusiond}).
Remarkably,  Diananda~\cite{Diananda:1962} and Maxfield and Minc~\cite{Maxfield:1962} have shown, respectively, that $\cp^{n*} = \dnn^{n*}$ and  $\cp^n = \dnn^n$ for any $n \leq 4$.
Hence equality holds throughout in (\ref{eqinclusion}) and (\ref{eqinclusiond}) for $n\le 4$. Moreover the inclusions 
$\cp^n\subseteq \dnn^n$ and $\dnn^{n*}\subseteq \cp^{n*}$ are known to be strict for any $n\ge 5$.
It suffices to show the strict inclusions for $n=5$, since  $A\in \dnn^5\setminus \cp^5$ implies $\widetilde A \in \dnn^n\setminus \cp^n$, where $\widetilde A$ is obtained by adding a border of zero entries to $A$. This extends to the cone $\cgpsd$. Indeed, the matrix $A$ belongs to $\cp^5$ (resp., $\dnn^5$, or $\cgpsd^5$) if and only if the extended matrix $\tilde A$ belongs to $\cp^n$ (resp., $\dnn^n$, or $\cgpsd^n$).

To show strict inclusions, we use some $5\times 5$  circulant matrices, with the following form:
\begin{equation*}\label{eqMcirc}
M(b,c) = \left( \begin{array}{rrrrr}
1 & b & c & c & b \\
b & 1 & b & c & c \\
c & b & 1 & b & c \\
c & c & b & 1 & b \\
b & c & c & b & 1
\end{array}\right), \ \text{ where } b,c\in\R.
\end{equation*}
The matrix $H = M(-1,1)$ is the well known Horn matrix, which is copositive  but  does not lie in the dual of the doubly nonnegative cone (see e.g.~\cite{Berman:2003}).
The strict inclusion: $\cp^5\subsetneq \cgpsd^5$ is shown in \cite{Fawzi:2014}
 using the matrix $L = M(\cos^2(4 \pi /5),\cos^2(2 \pi /5))=M((3+\sqrt 5)/8, (3-\sqrt 5)/8)$.
To show that $L$ is completely psd,
consider the matrix $\widehat L=M(\cos(4\pi/5),\cos(2\pi/5))$, so that $L$ is the entrywise square of $\widehat L$.
 Then, the vectors $x_i=(\cos(4i\pi/5), \sin(4i\pi/5))\in\R^2$ (for $i\in [5]$) form a Gram representation of $\widehat L$ and thus the psd matrices
 $x_ix_i^T\in \mathcal S^2_+$ ($i\in [5]$) form a Gram representation of $L$, which shows that $L$ is completely psd. 
However, $L$ is not completely positive, since its inner product with the Horn matrix is negative:
$\langle L,H\rangle = 5(2-\sqrt 5)/2<0$. 
Therefore,
 $$L\in \cgpsd^5\setminus \cp^5 \ \text{ and } H\in \cp^{5*}\setminus \cgpsd^{5*}.$$
 Frenkel and Weiner \cite{Frenkel:2014} showed that (a multiple of) the matrix $W = M((\sqrt 5-1)/2,0)$
  is doubly nonnegative but not completely psd, which
  proves the strict inclusion: $\cgpsd^{5} \subsetneq \dnn^{5}$.  
  Checking that $W$ is doubly nonnegative is simple and also that $W$ is not completely positive (since $\langle H,W\rangle = 5(2-\sqrt 5)<0$).  In Theorem \ref{lem:oddcycle} below, we prove that, for matrices whose pattern  of nonzero entries forms a cycle, being completely positive is equivalent to being completely psd. 
  We can use this result to show that  $W \not\in \cgpsd^{5}$ and to construct a class of matrices that are doubly nonnegative but not completely psd (see Lemma~\ref{lemclass}).
  Furthermore, at the end of this section, we will present the proof given in \cite{Frenkel:2014} that $W \not\in \cgpsd^{5}$ and we will show how we can use their ideas to prove that, in fact, $W$  does not belong to the closure of $\cgpsd^5$ (see Theorem~\ref{theoFW} and the discussion before it).
  This shows the strict inclusion:
$\mathrm{cl}(\cgpsd^5)\subsetneq \dnn^5$ and, by taking the duals  of both sides,  the strict inclusion:   $\dnn^{5*}\subsetneq \cgpsd^{5*}$ 
(Corollary~\ref{corA}).
%Furthermore, Lemma~\ref{lemclass} below will provide a different proof for  $W \not\in \cgpsd^{5}$ and a large class of matrices that are doubly nonnegative but not completely psd.
%To this end, we consider matrices whose pattern  of nonzero entries forms a cycle and we show that for such matrices being completely positive is equivalent to being completely psd (see Theorem \ref{lem:oddcycle} below).

\medskip
Given a matrix $A\in \mathcal S^n$, its {\em support graph} is the graph $G(A)=([n],E)$ where there is an edge $\{i,j\}$ when $A_{ij}\ne 0$ and $i\ne j$.
Moreover, the {\em comparison matrix} of $A$ is the matrix $C(A)\in \mathcal S^n$ with entries
$C(A)_{ii}=A_{ii}$ for all $i\in [n]$ and $C(A)_{ij}=-A_{ij}$ for all $i\ne j\in [n]$. We will use the following result characterizing completely positive matrices whose support graph is triangle-free.

\begin{theorem}\label{thm:cycledetcomparisonmatrix}\cite{Drew:1994} (see also~\cite{Berman:2003})
Let $A\in \mathcal S^n$ and assume  that its  support graph is triangle-free. Then, $A$ is completely positive if and only if its comparison matrix $C(A)$ is positive semidefinite.
\end{theorem}

\noindent
We have the following easy result  for matrices supported by bipartite graphs.

\begin{lemma}\label{lembip}
Let $A\in \mathcal S^n$ and assume that $G(A)$ is a bipartite graph. Then, $A\in \cgpsd^n$ if and only if $A\in \cp^n$.
\end{lemma}

\begin{proof}
Assume $A\in \cgpsd^n$; we show that $A\in \cp^n$ (the reverse implication holds trivially). Say, $X_1,\ldots,X_n\in \mathcal S^d_+$ form a Gram representation of $A$.
As $G(A)$ is bipartite, consider a bipartition of its vertex set as $U\cup W$ so that all edges of $G(A)$ are of the form $\{i,j\}$ with $i\in U$ and $j\in W$.
Now, observe that the matrices $X_i$ for $i\in U$, and $-X_j$ for $j\in W$ form a Gram representation of the comparison matrix $C(A)$. This shows that $C(A)\succeq 0$ and, in view of Theorem~\ref{thm:cycledetcomparisonmatrix}, $A\in \cp^n$.
 \end{proof}

\noindent
The above result  also follows from the known characterization of {\em completely positive graphs}. Recall that a graph $G$ is  {\em completely positive} if every doubly nonnegative matrix with support $G$ is completely positive. 
Therefore, for any matrix $A$,  $$\text{if } G(A) \text{ is completely positive then: } A\in \dnn^n\Longleftrightarrow A\in \cgpsd^n\Longleftrightarrow A\in \cp^n.$$
Kogan and Berman \cite{Kogan:1993} show that a graph $G$ is completely positive if and only if it does not contain an odd cycle of length at least 5 as a subgraph.
In particular, any bipartite graph is completely positive. 
Moreover, odd cycles (of length at least 5) are not completely positive graphs. For example, the above mentioned matrix $W=M((\sqrt 5-1)/2,0)$ 
(which has the 5-cycle as support)
is doubly nonnegative, but not completely positive since its inner product with the Horn matrix is negative:
$\langle W,H\rangle = 5(2-\sqrt 5)<0$.

\noindent
We will also use the following elementary result about psd matrices.

\begin{lemma}\label{prop:orthogblockpsd}
Let $A$ and $B$ be positive semidefinite matrices with block-form:
$$ A = \begin{pmatrix}
A_1 & A_2 \\ A^T_2 & A_3 \end{pmatrix}\ \ \text{ and } \ \ 
B = \begin{pmatrix}
B_1 & B_2 \\ B^T_2 & B_3
\end{pmatrix},
$$
where $A_i$ and $B_i$ have the same dimension.
 If $\langle A, B \rangle = 0$ then  $\langle A_1,B_1 \rangle = \langle A_3,B_3 \rangle = - \langle A_2,B_2 \rangle$.
\end{lemma}

\begin{proof}
As $A,B\succeq 0$, $\langle A,B \rangle = 0$ implies  $A B = 0$ and thus  $A_1B_1 + A_2 B_2^T=0$ and $A_2^T B_2+ A_3 B_3 = 0$. Taking the trace we obtain the desired identities.
 \end{proof}

We can now characterize the completely psd matrices supported by a cycle.

\begin{theorem}\label{lem:oddcycle}
Let $A\in \mathcal S^n$ and assume that $G(A)$ is a cycle.
Then, $A \in \cgpsd^n$ if and only if $A \in \cp^n$. 
\end{theorem}

\begin{proof}
One direction is obvious since $\cp^n\subseteq \cgpsd^n$. Assume now that $A\in \cgpsd^n$ with $G(A)=C_n$ and let $X^1,\ldots, X^n\in \mathcal S^d_+$ be a psd  Gram representation  of $A$.
We will show that  $A\in \cp^n$. We consider only the non-trivial case when $n \geq 5$.
In view of Theorem \ref{thm:cycledetcomparisonmatrix}, it suffices to show that the comparison matrix $C(A)$ is positive semidefinite.

We first consider the easy case when  $n$ is even. 
%Let $X^1,\ldots, X^n\in \mathcal S^d_+$ be a psd  Gram representation  of $A$.
Then, 
%(as in the above proof of Lemma \ref{lembip}),
 the matrices $Y^1=-X^1,Y^2=X^2, Y^3=-X^3, Y^4=X^4, \ldots, Y^{n-1}=-X^{n-1}, Y^n=X^n$ form a Gram representation of $C(A)$, thus showing that $C(A)\succeq 0$ and concluding the proof. 
Notice that, since even cycles  are bipartite graphs, this case follows from Lemma \ref{lembip}. 
%\textcolor{red}{We gave again the argument here to highlight where the difficuty will be in the case when $n$ is odd.}

%
%Let $X^1,\ldots, X^n\in \mathcal S^d_+$ be a psd  Gram representation  of $A$.
%If $n$ is even, then 
%(as in the above proof of Lemma \ref{lembip}), the matrices $Y^1=-X^1,Y^2=X^2, Y^3=-X^3, Y^4=X^4, \ldots, Y^{n-1}=-X^{n-1}, Y^n=X^n$ form a Gram representation of $C(A)$, thus showing that $C(A)\succeq 0$ and concluding the proof in the case $n$ even. 

Now suppose that $n$ is odd. As we will see, in order to  construct a Gram representation of $C(A)$, we can choose the same  matrices $Y^1,\ldots, Y^{n-1}$  as above but we need to look in  more detail into the structure of the $X^i$'s in order to be able to tell how to define the last matrix $Y^n$.

For this, we now show that  the matrices $X^1,\ldots, X^n$ can be assumed to be $(n-2)\times (n-2)$ block-matrices, where we denote the blocks of $X^k$ as 
$X^k_{rs}$ for $r,s\in [n-2]$ (with $X^k_{sr}=(X^k_{rs})^T$) and the index sets of the blocks as $I_1,\ldots, I_{n-2}$.
Indeed, without loss of generality we can assume that $X^1=\left(\begin{matrix} X^1_{11}& 0\cr 0 & 0 \end{matrix}\right)$ where $X^1_{11}$ is positive definite and the index set of $X^1_{11}$ defines the index set $I_1$ of the first block.
Next, $X^2$ has the form $\left(\begin{matrix} X^2_{11} & X^2_{12} & 0\cr X^2_{21} & X^2_{22} & 0 \cr 0 & 0 & 0\end{matrix}\right)$, where 
$X^2_{22}\succ 0$ 
 and its index set defines the index set $I_2$ of the second block.
Next, we can write $X^3=\left(\begin{matrix} X^3_{11} & X^3_{12} & X^3_{13} & 0\cr 
X^3_{21} & X^3_{22} & X^3_{23} & 0\cr
X^3_{31} & X^3_{32} & X^3_{33} & 0\cr
0 & 0 & 0 & 0 \cr\end{matrix}\right)$, where $X^3_{33} \succ 0$ and $I_3$ is the index set of $X^3_{33}$. Hence $X^3$ has its blocks indexed by $I_1,I_2,I_3$ and $[d] \backslash (I_1 \cup I_2 \cup I_3)$.
Iteratively, for each $k\in \{2,3,\ldots, n-3\}$, the matrix $X^k$ has blocks $X^{k}_{r,s}$ for $r,s \in [k]$ with $X^{k}_{k,k} \succ 0$ and it has zero entries outside of these blocks.
%its zero elsewhere. 
The index sets of the blocks $X^k_{kk}$ for $1 \le k \le n-3$ define the sets $I_1,I_2,\dots I_{n-3}$ and the set $I_{n-2} := [d] \backslash (I_1 \cup I_2 \dots \cup I_{n-3})$ collects all the remaining indices.

By looking at the zero-pattern of the matrix $A$, we now show some structural properties of the $X^{k}$ matrices and that 
 each set $I_{k}$ is non-empty. 
As $A_{12} \neq 0$, we know that $I_{1} \neq \emptyset$.
Since $A_{13}=0$ we can conclude  that $X^3_{11}=0$ (and thus 
$X^3_{12}=X^3_{13}=0$) and that the only non-zero blocks of $X^{3}$ are $X^3_{22}, X^3_{23}, X^3_{32}$ and $X^3_{33}$.
Moreover, as $A_{23} \neq 0$ we get that $I_{2} \neq \emptyset$.
With the same reasoning, for each $k\in \{2,3,\ldots, n-3\}$, as $A_{k',k} = 0$ for all $k' \in [k-2]$ we have that all blocks of the matrix $X^k$ are equal to zero except its blocks
$X^k_{k-1,k-1}$, $X^k_{k-1,k}$, $X^k_{k,k-1}$  and $X^k_{kk}$.
Moreover the fact that, for all $k\in [n-2]$, $A_{k,k+1} \neq 0$
%the ($k$,$k+1$)-entry of $A$ is non-zero
 implies that the index set $I_{k}$ is non-empty. 
Additionally, using the fact that $A_{n-2,k} =0$ for each $k \in \{1,\dots,n-4\}$ we obtain that each block $X^{n-2}_{kk}$ is equal to zero.
Similarly, $X^{n-1}_{kk}$ is the zero matrix for every $k \in \{1,\dots,n-3\}$ as $A_{n-1,k} =0$.
For the matrix $X^n$ we cannot make any consideration on the presence of zero blocks.

%As $A_{13}=0$, we can conclude  that $X^3_{11}=0$ (and thus 
%$X^3_{12}=X^3_{13}=0$). 
%Hence $X^3$ has its blocks indexed by $I_1,I_2,I_3$ and $[d] \backslash (I_1 \cup I_2 \cup I_3)$ with all zero blocks except $X^3_{22}, X^3_{23}, X^3_{32}$ and $X^3_{33}$.
%Iterating, we have that, for each $k\in \{2,3,\ldots, n-3\}$, all blocks of the matrix $X^k$ are equal to zero except its blocks
%$X^k_{k-1,k-1}$, $X^k_{k-1,k}$, $X^k_{k,k-1}$  and $X^k_{kk}$ and moreover $X^k_{kk} \succ 0$. 
%Using the fact that $\langle X^{n-2}, X^k \rangle =0$ for each $k \in \{1,\dots,n-4\}$ we obtain that each block $X^{n-2}_{kk}$ is equal to zero.
%Similarly $X^{n-1}_{kk}$ is the zero matrix for every $k \in \{1,\dots,n-3\}$ as $\langle X^{n-1}, X^k \rangle =0$.
%For the matrix $X^n$ we cannot make any consideration on the presence of zero blocks.

We now indicate how to construct the (non-symmetric) matrix $Y^n$ from $X^n$: we just  change signs to its two blocks $X^n_{n-3,n-2}$ and $X^n_{n-2, n-2}$. In other words, we let $Y^n$ be the $(n-2)\times (n-2)$ block matrix with blocks $Y^n_{n-3,n-2}=-X^n_{n-3,n-2},$ $Y^n_{n-2,n-2}=-X^n_{n-2,n-2}$
and $Y^n_{rs}=X^n_{rs}$ for all other blocks. Let us stress that in particular we do not change the sign of the block $X^n_{n-2,n-3}$.
As in the case when $n$ is even, for any $1\le i\le n-1$, we set $Y^i=-X^i$ for odd $i$  and $Y^i=X^i$ for even $i$.

We claim that $Y^1,\ldots,Y^n$ form a Gram representation of the comparison matrix $C(A)$.
It is clear that $\langle Y^i,Y^j\rangle = C(A)_{ij}$ for all $i,j\in [n-1]$ and that   $\langle Y^1,Y^n \rangle =-A_{1n}=C(A)_{1n}$ and $\langle Y^i,Y^n \rangle=0$ for $2\le i\le n-3$ (since the blocks indexed by $[n-3]$ in $Y^n$ are the same as in $X^n$ and each block $Y^i_{r, n-2}$ is equal to 0).
Moreover, $\langle Y^n,Y^n\rangle=\langle X^n,X^n\rangle = C(A)_{nn}$  and
$\langle Y^{n-1},Y^n \rangle = -A_{n-1,n}=C(A)_{n-1,n}$.
Finally, we use Lemma \ref{prop:orthogblockpsd} to verify that $\langle Y^{n-2},Y^n \rangle =0$.
Indeed, we have that 
$$0= A_{n-2,n}=\langle X^{n-2},X^n\rangle = 
\langle
 \left(\begin{matrix} X^{n-2}_{n-3,n-3} &  X^{n-2}_{n-3,n-2}\cr X^{n-2}_{n-2, n-3} & X^{n-2}_{n-2,n-2}\end{matrix}\right),
\left(\begin{matrix} X^{n}_{n-3,n-3} &  X^{n}_{n-3,n-2}\cr X^{n}_{n-2, n-3} & X^{n}_{n-2,n-2}\end{matrix}\right)
\rangle$$
which, by  Lemma  \ref{prop:orthogblockpsd}, implies that 
$\langle X^{n-2}_{n-3,n-3}, X^{n}_{n-3,n-3}\rangle = \langle X^{n-2}_{n-2,n-2},  X^{n}_{n-2,n-2}\rangle.$ 
Therefore, $$\langle Y^{n-2},Y^{n}\rangle = 
 \langle
 \left(\begin{matrix} -X^{n-2}_{n-3,n-3} &  -X^{n-2}_{n-3,n-2}\cr -X^{n-2}_{n-2, n-3} & -X^{n-2}_{n-2,n-2}\end{matrix}\right),
\left(\begin{matrix} X^{n}_{n-3,n-3} & - X^{n}_{n-3,n-2}\cr X^{n}_{n-2, n-3} & -X^{n}_{n-2,n-2}\end{matrix}\right)
\rangle$$
is equal to 0.
 \end{proof}

As an application of Lemma \ref{lembip} and of Theorem \ref{lem:oddcycle},
we obtain:

\begin{lemma}
If $G$  is bipartite or an odd cycle, then $\vartheta ^{\cgpsd}(\overline G)=\alpha(\overline G)$ and $\Theta^{\cgpsd}(\overline G)=\chi_f(\overline G)$.
\end{lemma}

\begin{proof}
It suffices to show $\vartheta^{\cgpsd}(\overline G)\le \alpha(\overline G)$ (as the reverse inequality is clear).
For this pick a matrix $X\in \cgpsd$ feasible for the program defining $\vartheta^{\cgpsd}(\overline G)$. Then the support of $X$ is contained in $ G$ and thus is bipartite or an odd cycle.
By Lemma \ref{lembip} and Theorem \ref{lem:oddcycle},  $X$ is completely positive. Using Theorem \ref{theodeK}, this implies $\alpha(\overline G)\ge \langle J,X\rangle$ and thus
$\alpha(\overline G)\ge \vartheta^{\cgpsd}(\overline G)$.
The equality $\Theta^{\cgpsd}(\overline G)=\chi_f(\overline G)$ follows analogously using 
Theorem \ref{theoDR}.
\end{proof}

Using Theorem \ref{lem:oddcycle}, we can conclude that the matrix $W=M((\sqrt 5-1)/2,0)$  is not completely psd (since its support graph is $C_5$ and $W\not\in \cp^5$).
Moreover, using a result from Hamilton-Jester and Li \cite{Hamilton:1996} (Theorem \ref{theoHam} below) we can construct in Lemma \ref{lemclass} a class of instances in $\dnn^5\setminus \cgpsd^5$.

\begin{theorem}\cite{Hamilton:1996}\label{theoHam}
(i) Assume $n\ge 5$ is odd and consider a matrix $A \in \dnn^n$ with rank $n-2$. Then, $A$ lies on an extreme ray of $\dnn^n$ if and only if $G(A)=C_n$.
(ii) Moreover, if $A$ lies on an extreme ray of $\dnn^5$, then $A$ has rank 1 or 3.
\end{theorem}

\begin{lemma}\label{lemclass}
For odd $n\ge 5$, any matrix $A\in\mathcal S^n$ with support $G(A)=C_n$ and 
with rank $n-2$ is not completely positive semidefinite.
Moreover, for $n=5$, if $A\in\mathcal S^5$ lies on an extreme ray of $\dnn^5$, then $A$  is completely positive semidefinite if and only if $A$ is completely positive.
\end{lemma}

\begin{proof} 
Suppose that a matrix $A \in \cgpsd^{n}$ has support $G(A) = C_{n}$ and rank $n-2$, for some odd  $n \ge 5$.
By Theorem \ref{lem:oddcycle} we know that $A \in \cp^{n}$. Moreover, from Theorem \ref{theoHam}(i), $A$ lies on an extreme ray of  $\dnn^n$ and thus also of $\cp^n$.
Since the extreme rays of $\cp$ are rank 1 matrices, we get the contradiction: $1=\rank A=n-2$. Hence, $A \notin \cgpsd^{n}$.

 For the second claim, assume that $A$ lies on an extreme ray of $\dnn^{5}$. We show that if $A\not\in\cp^5$ then $A\not\in \cgpsd^5$ (the reverse implication is clear).
By Theorem \ref{theoHam}(ii), any matrix on an extreme ray of $\dnn^5$ has rank 1 or 3.
 Recall that a doubly nonnegative matrix which has rank 1 is  completely positive. Hence, if $A \notin \cp^{5}$ then $A$ must have rank 3 and support $G(A) = C_{5}$ (by Theorem \ref{theoHam}(i)).
 Now using the first part of the lemma we can conclude that $A\not\in \cgpsd^5$.
 \end{proof}

%\noindent
We now show that the matrix $W=M((\sqrt 5-1)/2,0)$ does not belong to the closure of $\cgpsd^5$. For this we use results from \cite{BLP} and \cite{Frenkel:2014}, dealing with Gram representations by positive elements in a  general  {\em finite von Neumann algebra} (an infinite dimensional analogue of Gram representations by positive semidefinite matrices). 
For our treatment, we only need to know that a finite von Neumann algebra $\mathcal N$  is  equipped 
with a trace $\tau$ (an analogue of the usual matrix trace) which satisfies the following properties:  for any $A,B\in \mathcal N$, 
$\tau(AB)=\tau(BA)$; % for all $A,B\in \mathcal N$,  
if $A$ is positive then $\tau (A)\ge 0$ with equality if and only if $A=0$;
and
%From these properties it follows that 
\begin{equation}\label{eqtau2}
\text{if } A,B \text{ are positive and  }  \tau (AB)=0   \text{ then }  AB=0.
\end{equation}
On the one hand, it is shown in \cite{BLP} that  there exists a finite von Neumann algebra $\mathcal M$ (with trace $\tau$) with the property that  any matrix $X$ lying in the closure of $\cgpsd^n$ admits a Gram representation by positive elements of $\mathcal M$, i.e., $X=(\tau(A_iA_j))_{i,j=1}^n$ for some positive $A_1,\ldots,A_n\in\mathcal M$.
On the other hand,  it is shown in \cite{Frenkel:2014}  that the matrix $W$ does not admit a Gram representation by positive elements in any finite von Neumann algebra, see Theorem \ref{theoFW} below. Hence, by combining these two results,  we can conclude that the matrix $W$ does not belong to the closure of $\cgpsd^5$.

\begin{theorem} \cite{Frenkel:2014}\label{theoFW}
Let $\mathcal N$ be a finite von Neumann algebra with trace $\tau$. For the matrix $W=M((\sqrt 5-1)/2,0)$, there do not exist positive elements $A_1,\ldots,A_5\in \mathcal N$ such that $W=(\tau(A_iA_j))_{i,j=1}^5$.
\end{theorem}

\begin{proof}
We give a  proof for completeness. Assume  that $W=(\tau(A_iA_j))_{i,j=1}^5$ for some positive $A_1,\ldots,A_5\in \mathcal N$.
Using (\ref{eqtau2}), $W_{i,i+2}=0$ implies $A_iA_{i+2}=0$ for all $i\in [5]$ (taking indices modulo 5).
As $W$ is a rank 3 positive semidefinite matrix, %doubly nonnegative matrix, 
 there exist vectors $u_1,\ldots,u_5\in\R^3$ forming a Gram representation of $W$ and the set $\{u_1,\ldots,u_5\}$ has rank 3. Moreover, one can check  that the set $\{u_1,u_3,u_4\}$ is a base of $\R^3$. 
Hence, $u_2=\alpha u_1+\beta u_3+\gamma u_4$ for some $\alpha,\beta,\gamma\in\R$. Using the fact that $W=(u_i^T u_j)_{i,j=1}^5=(\tau(A_iA_j))_{i,j=1}^5$, we obtain the analogous relation:
$A_2=\alpha A_1+\beta A_3+\gamma A_4$. Multiplying both sides by $A_1$ gives $A_1A_2=\alpha A_1^2$. 
Analogously, expressing $u_1$ in the base $ \{u_2,u_4,u_5\}$ implies that $A_1A_2=\alpha' A_2^2$ for some $\alpha'\in \R$. 
Thus, $\alpha A_1^2=\alpha'A_2^2$, implying $\alpha=\alpha'$ (since $W_{11}=W_{22}=1$) and thus $A_1=A_2$, a contradiction (since $W_{12}\ne 1$).
\end{proof}

\begin{corollary}\label{corA}
The inclusions  $\mathrm{cl}(\cgpsd^n)\subseteq \dnn^n$ and 
$\dnn^{n*}\subseteq \cgpsd^{n*}$ are  strict for any $n\ge 5$.
\end{corollary}

\subsection{The dual cone of the completely positive semidefinite cone }\label{secdual}

The dual of the completely positive cone $\cp^n$ is the copositive cone $\mathcal{COP}^n$, consisting of the matrices $M\in \mathcal S^n$ 
for which  the $n$-variate polynomial 
$p_M=\sum_{i,j=1}^n M_{ij}x_i^2x_j^2$ is nonnegative over $\R^n$, i.e., $\sum_{i,j=1}^n M_{ij}x_i^2x_j^2\ge 0$ for all $x_1,\ldots,x_n\in \R^n$. We now consider the dual of the cone $\cgpsd^n$.

\begin{lemma}\label{lemdual}
Given a matrix $M\in \mathcal S^n$, the following assertions are equivalent:
\begin{itemize}
\item[(i)] $M\in \cgpsd^{n*}$, i.e., $\sum_{i,j=1}^n M_{ij} \langle X_i,X_j\rangle \ge 0$ for all $X_1,\ldots, X_n\in \mathcal S^d_+$ and  $d\in \N$.
\item[(ii)]
$\Tr(\sum_{i,j=1}^n M_{ij} X_i^2 X_j^2)\ge 0$ for all $X_1,\ldots, X_n\in \mathcal S^d$ and $d\in \N.$
\end{itemize}
\end{lemma}

\begin{proof}
Use the fact that any  matrix $X\in \mathcal S^d_+$ can be written as $X=Y^2$ for some $Y\in \mathcal S^d$. Indeed, write $X=PDP^T$, where $P$ is orthogonal and $D$ is the diagonal matrix containing the eigenvalues of $X$, and set $Y=P\sqrt DP^T$.
  \end{proof}

\noindent
In other words,  $M\in \cgpsd^{n*}$ if the associated polynomial $p_M= \sum_{i,j=1}^n M_{ij} X_i^2 X_j^2$  in the non-commutative variables $X_1,\ldots,X_n$ is {\em trace positive}, which means that  the  evaluation of $p_M$  at any symmetric matrices $X_1,\ldots, X_n$ (of the same arbitrary  size $d\ge 1$) produces a matrix with nonnegative  trace.
Hence copositivity corresponds to  restricting to symmetric matrices $X_i$ of  size $d=1$, i.e., to real numbers.

Interestingly, understanding which matrices lie in $\cgpsd^{n*}$ is deeply connected with Connes' embedding conjecture~\cite{Connes:1976}, one of the most important conjectures in von Neumann algebra.
A reformulation of the conjecture that shows this connection is given by Klep and Schweighofer~\cite{Klep:2008}, see Conjecture~\ref{conj:equivconnes} below. In order to state it, we need to introduce some notation.

\medskip
We let $\mathbb{R}[\underline{x}]$ (resp., $\mathbb{R}\langle \underline{X} \rangle$)  denote
 the set of real polynomials in the commutative variables $x_1,\dots,x_n$ (resp.,  in the non-commutative variables $X_1,\dots,X_n$).
$\mathbb{R}\langle \underline{X} \rangle$ is endowed with the involution $*: \mathbb{R}\langle \underline{X} \rangle \to \mathbb{R}\langle \underline{X} \rangle$ that sends each variable to itself,
 each monomial $X_{i_1}X_{i_2}\cdots X_{i_t}$ to its reverse $X_{i_t}\cdots X_{i_2}X_{i_1}$ and extends linearly to arbitrary polynomials; e.g.,
 $(X_1X_2 + X_2X_3^2)^* = X_2X_1 + X_3^2X_2$.
A polynomial $f\in \ncP$ is {\em symmetric} if $f^*=f$ and $\mathit{S}\mathbb{R}\langle \underline{X} \rangle$ denotes  the set of symmetric  polynomials in $\ncP$. 
A polynomial of the form $ff^*$ is called a {\em Hermitian square} and a polynomial of the form $[f,g]=fg-gf$ is called a  {\em commutator}.
A polynomial $f\in \ncP$ is said to be {\em trace positive} if $\Tr(f(X_1,\ldots, X_n))\ge 0$ for all $(X_1,\ldots,X_n)\in \cup_{d\ge 1} (\mathcal S^d)^n$.
 Note that $f^*$ evaluated at $(X_1,\ldots,X_n)\in (\mathcal S^d)^n$ is equal to $f(X_1,\ldots,X_n)^T$; hence, any Hermitian square $ff^*$ is trace positive.
 Moreover, the trace of any commutator vanishes when evaluated at any $n$-tuple of matrices.

\medskip
  The {\em tracial quadratic module} $\text{\rm tr}\M$ generated by a set of polynomials $p_1,\ldots,p_m\in S\ncP$ consists of all polynomials of the
form $h+ \sum_{j=1}^{m_0} f_jf_j^* +\sum_{i=1}^m \sum_{j_i=1}^{m_i} g_{j_i} p_i g_{j_i}^*$, where $h\in \ncP$ is a sum of commutators,   $ f_j,g_{j_i}\in \ncP$ and $m_0,m_i\in \N$.
 We consider  here  the tracial quadratic module $\trMcube$  generated by  the polynomials $1-X_1^2, \ldots, 1-X_n^2$, 
 and the tracial quadratic module $\trMball$  generated by the polynomial $1-\sum_{i=1}^n X_i^2$. 
  Clearly any polynomial in $\trMcube$ (resp., in $\trMball$) is trace positive on the (non-commutative version of the) hypercube
  $\Qnc$ (resp., on the non-commutative ball $\Bnc$), where we set
 \begin{equation*}\label{eqcube}
\Qnc=\bigcup_{d\ge 1} \left\{ (X_1,\ldots,X_n)\in (\mathcal S^d)^n: I-X_i^2\succeq 0\ \forall i\in [n] \right\},
\end{equation*}
\begin{equation*}\label{eqball}
\Bnc = \bigcup_{d\ge 1} \left\{ (X_1,\ldots,X_n)\in (\mathcal S^d)^n: I-\sum_{i=1}^n X_i^2\succeq 0\right \}.
\end{equation*}
\noindent

Klep and Schweighofer \cite{Klep:2008} (see also \cite{Burgdorf:2013}) showed that Connes' embedding conjecture 
is equivalent to the following  conjecture   characterizing the trace positive polynomials on $\Qnc$. 

\begin{conjecture}\label{conj:equivconnes}\cite{Klep:2008} 
Let $f \in \mathit{S}\mathbb{R}\langle \underline{X} \rangle$. The following are equivalent:
\begin{enumerate}[label=\roman*.]
\item[(i)] \label{enu:conj1} $f$ is trace positive on $\Qnc$, i.e., $\Tr(f(X_1,\ldots,X_n))\ge 0$ for all $(X_1,\ldots,X_n)\in \Qnc$.
\item[(ii)]\label{enu:conj2} For any  $\epsilon > 0$, $f + \epsilon\in \trMcube$, i.e., $f + \epsilon= g+h$, where  $h$ is a sum of commutators and $g= \sum_{j=1}^{m_0} f_jf_j^* +\sum_{i=1}^n \sum_{j_i=1}^{m_i} g_{j_i} (1-X_i^2) g_{j_i}^*$
for some $f_j,g_{j_i}\in \ncP$ and $m_0,m_i\in \N$.
\end{enumerate}
\end{conjecture}

\noindent
In fact, Connes' embedding conjecture is also equivalent to Conjecture \ref{conj:equivconnes} where we restrict $f$ to have degree at most 4 (see \cite[Proposition 2.14]{Burgdorf:2011}).
Note that the polynomials $p_M$  involve only monomials of the form $X_i^2X_j^2$.
 Interestingly, in the proof that Conjecture \ref{conj:equivconnes} is equivalent to Connes' embedding conjecture,  these monomials  $X_i^2X_j^2$ play a fundamental role (due to a result of R\u{a}dulescu~\cite{Radulescu:1999}). 
Finally, let us point out that, as observed by Burgdorf \cite[Remark 2.8]{Burgdorf:2011}, Connes' conjecture is also equivalent to Conjecture~\ref{conj:equivconnes}  where  the ball is used instead of the hypercube, i.e., replacing the tracial quadratic module $\trMcube$  by the tracial  quadratic module $\trMball$.

\medskip
While Conjecture \ref{conj:equivconnes} involves trace positive polynomials on the hypercube, membership of a matrix $M$ in $\cgpsd^{n*}$ requires that the polynomial $p_M$  is trace positive on {\em all} symmetric matrices. To make the link between both settings, the key (easy to check) observation  is that, since  $p_M$ is  a homogeneous polynomial,  
 trace positivity over the hypercube, over the full space and over the ball are all equivalent properties. This gives:

 \begin{lemma}\label{lemball}
A matrix $M\in \mathcal S^n$ belongs to $\cgpsd^{n*}$ if and only if the associated polynomial $p_M$ is trace positive over  the cube $\Qnc$ or, equivalently, over  the ball $\Bnc$.
\end{lemma}

\subsection{Approximating the dual cone of $\cgpsd$}\label{secapprsets}

For a matrix $M\in \mathcal S^n$,
if  its associated polynomial $p_M$ belongs to the tracial quadratic module $\trMball$ then $M$  belongs to the dual cone $\cgpsd^{n*}$. We now define the set $\Knceps$ consisting of all matrices $M$ for which the perturbed polynomial $p_M+\epsilon$ belongs to $\trMball$. To simplify the notation, in   $\Knceps$ we omit the dependence on the size $n$ of the matrices.

\begin{definition}\label{defKnc}
For $\epsilon\ge 0$, let $\Knceps$ denote the set of matrices $M\in \mathcal S^n$ for which the polynomial $p_M+\epsilon$ belongs to the tracial quadratic module $\trMball$.

\end{definition}

\begin{lemma}\label{lembasic}
For any $\epsilon \geq 0$, $\Knceps$ is a convex set. 
Moreover, we have inclusion $\displaystyle \bigcap_{\epsilon>0} \Knceps\subseteq \cgpsd^{n*}$,
with equality if Connes' embedding conjecture holds.
\end{lemma}

\begin{proof}
Convexity follows from the fact that $p_{\lambda M+(1-\lambda)M'}+\epsilon=\lambda (p_M+\epsilon)+(1-\lambda)(p_{M'}+\epsilon)$  for $M,M'\in\mathcal S^n$ and $\lambda\in [0,1]$. 
Assume $M\in  \bigcap_{\epsilon>0} \Knceps$.
Then, for any $\epsilon>0$, the polynomial $p_M+\epsilon$ is trace positive on the ball $\Bnc$. By letting $\epsilon$ tend to 0, we obtain that $p_M$ is trace positive on $\Bnc$ and thus $M\in \cgpsd^{n*}$ by Lemma \ref{lemball}.
Finally, equality $ \bigcap_{\epsilon>0} \Knceps= \cgpsd^{n*}$ holds under Connes' embedding conjecture since, as mentioned above,  by results of  \cite{Klep:2008,Burgdorf:2013}  Connes' embedding conjecture  is equivalent to 
Conjecture~\ref{conj:equivconnes}, also when  the ball is used instead of the hypercube.
\end{proof}

We now point out a connection between the set $\Knceps$ and the following set $\Kc$, used in the commutative setting.
Let $\Sigma$ denote the set of  sums of squares of (commutative) polynomials and following \cite{Parrilo:2000} define the cone
\begin{equation}\label{defKc}
\Kc:=\left\{M\in\mathcal S^n: \exists r\in \N \ p_M(\sum_{i=1}^nx_i^2)^r \in\Sigma\right\}= \left\{M\in\mathcal S^n: p_M\in \Sigma +(1-\sum_{i=1}^n x_i^2)\R[\underline x]\right\}
\end{equation}
(see \cite[Proposition 2]{deKlerk:2005} for the equivalence between both definitions).
The inclusion  $\Kc \subseteq \cop$ is clear and Parrilo \cite[Section 5.3]{Parrilo:2000} showed that $\Kc$ covers the interior of $\cop$.
Moreover, by adding degree constraints on the terms entering the decomposition of $p_M$, he defined a hierarchy of 
subcones of $\cop$, whose first level is equal to the dual of the doubly nonnegative cone~\cite[Section 5.3]{Parrilo:2000}:
$$\Kc^{(0)} := \left\{M\in\mathcal S^n: p_M\in \Sigma\right\}= \mathcal S^n_+ + (\mathcal S^n \cap \R^{n\times n}_+)= \dnn^{n*}.$$
It turns out that the set $\K_{\text{\rm nc},0}$ is equal to $\Kc^{(0)}$.
 
\begin{lemma}\label{lemBall}
We have: $\dnn^{n*} = \Kc^{(0)}=\K_{\text{\rm nc},0}\subseteq \Knceps$ for any $\epsilon > 0$.
\end{lemma}

\begin{proof} The inclusion $\K_{\text{\rm nc},0}\subseteq \Knceps$ is clear.

First we show the inclusion $\K_{\text{\rm nc},0} \subseteq \Kc^{(0)}$. For this, assume $M\in \K_{\text{\rm nc},0}$, i.e.,
$p_M=h+g$, where $h$ is a sum of commutators and
$g= \sum_{j=1}^{m_0} f_jf_j^* +\sum_{j=1}^{m_1} g_j (1-\sum_{i=1}^nX_i^2) g_j^*$ with $f_j,g_j\in \ncP$.
If we evaluate $p_M$ at commutative variables $x$, we see that $h(x)$ vanishes and thus we obtain $p_M(x)=g(x) \in \Sigma+(1-\sum_{i=1}^nx_i^2)\Sigma$.
As $p_M$ is a homogeneous polynomial, %of degree 4
 we can derive %(using  \cite[Prop. 4]{deKlerk:2005})
 that $p_M\in \Sigma$ and thus $M\in \Kc^{(0)}$.
This follows from  \cite[Proposition 4]{deKlerk:2005} which shows that, for a homogeneous polynomial of even degree, membership in $\Sigma +(1-\sum_{i=1}^nx_i^2)\Sigma$ implies  membership in $\Sigma$.

We now show the reverse inclusion $ \Kc^{(0)} \subseteq\K_{\text{\rm nc},0}$. As $\Kc^{(0)}=\dnn^{n*}=\mathcal S^n_+ +(\mathcal S^n \cap \R^{n\times n}_+)$, it suffices to show that if $M\succeq 0$ or if $M\ge 0$ then $p_M$ is a sum of commutators and of Hermitian squares, which implies that $M\in \K_{\text{\rm nc},0}$.
Assume first that $M\succeq 0$ and let $u_1,\ldots,u_n\in \R^d$ be vectors forming a Gram representation of $M$.
Then, $p_M(\underline X)=\sum_{i,j=1}^n \sum_{h=1}^d u_i(h)u_j(h)X_i^2X_j^2=\sum_{h=1}^d (\sum_{i=1}^n u_i(h)X_i^2)^2$ is a sum of Hermitian squares. 
Assume now that  $M\ge 0$. Then each  $M_{ij}X_i^2X_j^2=M_{ij}([X_i^2X_j,X_j]+X_jX_i^2X_j)$ is  sum of a commutator and a Hermitian square % (of degree 4).
and thus $p_M$ is sum of commutators and Hermitian squares. 
\end{proof}

We conclude with some remarks concerning how well $\Kc$ and $\Knceps$ approximate  the cones $\cop$ and $\cgpsd^{*}$, respectively. 
As mentioned above, Parrilo \cite{Parrilo:2000} showed that $\Kc$ covers the interior of the copositive cone, i.e., $\text{int}(\cop)\subseteq \Kc \subseteq \cop,$ which can also be derived using the following result of Schm\"udgen \cite{Schmudgen:1991}.

\begin{theorem} %[Schm\"udgen] 
\cite{Schmudgen:1991}  \label{theoSchmudgen}
If $f\in \R[\underline x]$ is positive on the sphere, i.e.,
$f(x)>0$ for all $x\in\R^n$ with $\sum_{i=1}^nx_i^2=1$, then $f\in \Sigma + (1-\sum_{i=1}^n x_i^2)\R[\underline x]$.
\end{theorem}

In the non-commutative case, membership of a matrix $M$ in  $\Knceps$ means that the polynomial $p_M+\epsilon$ belongs to the tracial quadratic module $\trMball$, but there is no clear link between this and membership in the interior of the cone $\cgpsd^{*}$.

To explain this difference of behavior between $\Kc$ and $\Knceps$ let us point out that, in the commutative (scalar) case,  working with the ball is in some sense equivalent to working with the sphere. Indeed, as $p_M$ is homogeneous, it is nonnegative over $\R^n$ if and only if it is nonnegative over the ball or, equivalently, over the sphere, because  one can {\em rescale} any nonzero $x\in \R^n$ so that  $\sum_{i=1}^n x_i^2~=~1$. However, when working with matrices $X_1,\ldots,X_n$, one can rescale them to ensure that $I-\sum_{i=1}^n X_i^2\succeq 0$ but one cannot ensure equality: $\sum_{i=1}^n X_i^2=I$. Hence, in the non-commutative case one cannot equivalently switch between the ball and the sphere.

\section{Conic programs for  the quantum graph parameters}\label{secconicopt}

In this section we  show how to reformulate the quantum graph parameters as conic optimization problems using the completely positive semidefinite cone $\cgpsd$. We first express each quantum graph parameter 
through testing the feasibility of a sequence of optimization programs over the cone  $\cgpsd$  (Propositions \ref{def:cgqindep} and \ref{def:cgqchrom}) and then as a single `aggregated' optimization program over $\cgpsd$ (Proposition \ref{propoptchiq}).
Moreover we show that, if in  these conic programs we   replace the cone $\cgpsd$ by its subcone $\cp$ or by  its supercone $\dnn$, then we  find respectively the classical graph parameters and their corresponding bounds in terms of the theta number (Corollaries \ref{cor:theta'equiv} and \ref{corqchi}). 
In Section \ref{secPsi}, we use the convex sets $\Knceps$ to define the new parameters $\Psi_\epsilon(G)$.

\subsection{Conic reformulation for quantum stability numbers}\label{secconealpha}

We begin with  providing an equivalent reformulation for the two quantum stability numbers $\alpha_q(G)$ and $\qindep(G)$ as conic feasibility programs over  the completely positive semidefinite cone $\cgpsd$.

\begin{proposition} %[Game entanglement-assisted stability number]
\label{def:cgqindep}
For a graph $G$,  the parameter $\alpha_q(G)$ is equal to  the maximum  $t\in \N$ for which  there exists a matrix $X \in \cgpsd^{|V(G)|t + 1}$ (indexed by  $V(G) \times [t] \cup \{0\}$) satisfying the following conditions: 
\begin{align}
&X_{0,0} = 1, \label{eq:cqalpha-1} \tag{C1} \\
& \sum_{u \in V(G)} X_{0,ui}= 1\quad \forall i\in [t]\label{eq:cqalpha-sum2a} \tag{C2a},\\
&\sum_{u,v \in V(G)} X_{ui,vi} =  1\quad \forall i\in [t] \label{eq:cqalpha-sum-2} \tag{C2b},\\
&X_{ui,vj} = 0\quad \forall i\ne j\in [t], \, \forall u\simeq v \in V(G), \label{eq:cqalpha-orth1} \tag{O1}\\
&X_{ui,vi} =0 \quad \forall i \in [t], \, \forall  u\neq v \in V(G) \label{eq:cqalpha-orth2} \tag{O2}.
\end{align}
Moreover, the parameter  $\qindep(G)$ is equal to  the maximum integer $t$ for which  there exists a matrix $X \in \cgpsd^{|V(G)|t + 1}$ satisfying  (\ref{eq:cqalpha-1}), (\ref{eq:cqalpha-sum2a}), (\ref{eq:cqalpha-sum-2}) and  (\ref{eq:cqalpha-orth1}).
\end{proposition}

\begin{proof}
Observe that if $X$ satisfies (\ref{eq:cqalpha-orth2}), then condition (\ref{eq:cqalpha-sum-2})  is equivalent to  the condition: $\sum_{u \in V(G)} X_{ui,ui} =1$ for all $i\in [t]$.
We first consider the parameter $\alpha_q(G)$. 

By Definition \ref{def:gqindep}, there exist positive semidefinite matrices $\rho, \rho^u_i$ (for $u\in V(G)$, $i\in [t]$) satisfying (\ref{eq:qalpha-1})-(\ref{eq:qalpha-ort2}).
Let $X$ denote the Gram matrix of $\rho, \rho^u_i$, i.e., $X_{0,0} = \langle \rho,\rho\rangle$, $X_{0,ui}=\langle \rho,\rho^u_i\rangle$ and $X_{ui,vj}=\langle \rho^u_i,\rho^v_j\rangle$ for all $u,v\in V(G)$, $i,j\in [t]$.  By construction,  $X$ belongs to the cone $\cgpsd^{|V(G)|t+1}$. Moreover, $X$   satisfies the conditions 
(\ref{eq:cqalpha-1}), (\ref{eq:cqalpha-orth1}) and (\ref{eq:cqalpha-orth2}) which correspond, respectively, to  (\ref{eq:qalpha-1}), (\ref{eq:qalpha-ort1}) and (\ref{eq:qalpha-ort2}).
Next, using (\ref{eq:qalpha-1}) and (\ref{eq:qalpha-sum}), we obtain that for any $i\in [t]$:
$1=\langle \rho,\rho\rangle =\langle \rho, \sum_u \rho^u_i\rangle= 
\langle \sum_u \rho^u_i,\sum_v\rho^v_i\rangle =\sum_u \sum_{v} \langle \rho^u_i,\rho^v_i\rangle $
which shows that $X$ also satisfies (\ref{eq:cqalpha-sum2a}) and  (\ref{eq:cqalpha-sum-2}).

Conversely, assume that  $X\in \cgpsd^{|V(G)|t+1}$ satisfies the conditions (\ref{eq:cqalpha-1}), (\ref{eq:cqalpha-sum2a}), (\ref{eq:cqalpha-sum-2}), (\ref{eq:cqalpha-orth1}) and (\ref{eq:cqalpha-orth2}). As   $X$ is completely positive semidefinite,  there exist positive semidefinite matrices $\rho,\rho^u_i$  forming a Gram representation of $X$; we show that the matrices $\rho,\rho^u_i$ satisfy the conditions of Definition \ref{def:gqindep}.
It  is clear that   (\ref{eq:qalpha-1}), (\ref{eq:qalpha-ort1}) and (\ref{eq:qalpha-ort2}) hold.
Next, for any $i\in [t]$, we have:
$\| \rho - \sum_{u \in V(G)} \rho_i^u \|^2 = 1 - 2 \sum_{u \in V(G)}X_{0,ui} + \sum_{u,v \in V(G)} X_{ui,vi} =0,$
using (\ref{eq:cqalpha-1}), (\ref{eq:cqalpha-sum2a}) and (\ref{eq:cqalpha-sum-2}). This shows  (\ref{eq:qalpha-sum})  and thus  concludes the proof for $\alpha_q(G)$.

The proof is analogous for the parameter $\qindep(G)$ and thus omitted. 
 \end{proof}

\noindent
Next we observe that, in Proposition \ref{def:cgqindep},  we can restrict without loss of generality to solutions that are invariant under action of the permutation group $\Sym(t)$ (consisting of all permutations of $[t] = \{1,\dots, t\}$). 
We sketch this well known symmetry reduction, which has been used in particular for the study of the chromatic number in \cite{Gvozdenovic:2008}.

Given  $Y \in \mathcal S^{|V(G)|t+1}$  and a permutation $\pi \in$ Sym($t$), define the new matrix $\pi(Y)$ with  entries $\pi(Y)_{00} = Y_{00}$, $\pi(Y)_{0,ui} = Y_{0,u\pi(i)}$ and  
$\pi(Y)_{ui,vj}= Y_{u\pi(i),v\pi(j)}$ for $i,j \in [t], u,v \in V(G)$,  and the matrix
$Y'= \frac{1}{\lvert \text{Sym}(t) \lvert} \sum_{\pi \in \text{Sym}(t)} \pi(Y)$, called the {\em symmetrization} of $Y$ under action of $\Sym(t)$.
Then, $Y'$ is invariant under action of $\Sym(t)$, i.e., $\pi(Y')=Y'$ for all $\pi\in \Sym(t)$, and thus   $Y'$ has the following block-form:
\beq \label{eq:block}
 \begin{pmatrix}
\alpha & a^T & a^T & \ldots & a^T \\
a & A & B &\ldots & B \\
a & B & A &\ldots & B \\
\vdots & \vdots & \vdots & \ddots & \vdots \\
a & B  & B  & \ldots & A
\end{pmatrix}\
\text{ for some } \alpha \in \R, a\in \R^{|V(G)|}, \ A,B \in \mathcal S^{|V(G)|}.
\eeq

\noindent
Notice that the programs described in Proposition \ref{def:cgqindep} are invariant under action of $\Sym(t)$; that is, if $Y$ is feasible for one of them then any permutation $\pi(Y)$ is feasible too and thus its symmetrization $Y'$ as well. 
Therefore both programs have a feasible solution in block-form (\ref{eq:block}) (assuming some exists).

This invariance property, which  holds not only for the cone $\cgpsd$ but also for the cones $\mathcal S_+$, $\cp$ and $\dnn$, will be useful, together with the following lemma, for proving Proposition \ref{lem:theta'equiv} below.

\begin{lemma}\label{lem:block}(see e.g. \cite{Gvozdenovic:2008})
Let $Y$ be a $t \times t$ block-matrix, of the form: 
\begin{equation} \label{eqblock}
Y = 
\underbrace{\begin{pmatrix}
A & B & \ldots & B \\
B & A & \ldots & B \\
\vdots & \vdots & \ddots & \vdots \\
B & B  & \ldots & A
\end{pmatrix}}_\text{t blocks},
\end{equation}
having $A$ as diagonal blocks and $B$ as off-diagonal blocks, where $A,B\in \mathcal S^k$ (for some $k\ge 1$).
Then, 
 $Y \succeq 0\Longleftrightarrow A - B \succeq 0 \text{ and } A + (t-1)B \succeq 0.$
\end{lemma}

Next we consider again  the programs introduced in Proposition \ref{def:cgqindep} for defining the parameters $\alpha_q(G)$ and $\qindep(G)$, and we investigate what is their optimum value when replacing the cone $\cgpsd$ by any of the two cones $\cp$ or $\dnn$. We show that when using $\cp$ we find the classical stability number $\alpha(G)$ while, when using the cone $\dnn$, we find the parameter $\lfloor \vartheta'(G)\rfloor$, see Corollary \ref{cor:theta'equiv} below. 

For this  we will need two preliminary results, which we will use in the proof of Proposition~\ref{lem:theta'equiv} below.
The first one (Theorem \ref{lem:2*2cp})  gives a  property of completely positive matrices, which we will apply later for the choice of $B$ 
 having {\small $\left(\begin{matrix} 1 & -1 \cr -1 & 1 \end{matrix}\right)$}
as its $2\times 2$ nonzero principal submatrix.

\begin{theorem}\cite{Barioli:2001}\label{lem:2*2cp}
Let $A,B \in \mathcal S^n$. 
Assume  that $A$ is completely positive, $B$ is positive semidefinite with all its entries equal to zero except for a $2 \times 2$ principal submatrix, and that $A + B$ is a nonnegative matrix. 
Then $A + B$ is completely positive.
\end{theorem}

Given a graph $G$ and an integer $t\ge 1$, 
we introduce the graph $G_t$ which models the  orthogonality conditions (\ref{eq:cqalpha-orth1}) and (\ref{eq:cqalpha-orth2}), i.e., its vertex set is $V(G) \times [t]$ and two distinct vertices are adjacent in $G_{t}$ if $i \neq j$ and $u\simeq v$, or if $i = j$ and $u \neq v$. 
% denote  the Cartesian product of $G$ and the complete graph $K_t$.
%we let $G_t=G\Box K_t$  denote  the Cartesian product of $G$ and the complete graph $K_t$.
The second result (Lemma \ref{lemthetap}) gives a property of the parameter $\vartheta'$ for the graph $G_t$.

\begin{lemma}\label{lemthetap}
Let $G$ be a graph and let  $t\ge 1$ be an integer such that $\vartheta'(G)\ge t$.
Then, $\vartheta'(G_t)\ge t$.
\end{lemma}

\begin{proof}
Let $X$ be a matrix which is an optimal solution for the program (\ref{eqthetas-K}) defining $\vartheta'(G)$, meaning that $\langle J,X\rangle =\vartheta'(G)$.
 We set $n=|V(G)|$ and $T=\vartheta'(G)$.
Define  the diagonal matrix $D\in\mathcal S^n$ with  $D_{uu}=X_{uu}$  for all $u\in V$ and
 the matrix $M=  (T-1)D\otimes I_t - (D-X)\otimes (J_t-I_t)$ in $\mathcal S^{nt}$.
 Then, $M$ is entrywise nonnegative, its entries  are zero at all positions corresponding to edges of $G_t$,
$\Tr(M)= (T-1)t$, and $\langle J,M\rangle = (T-1)t^2$. Hence, if we can show that $M\succeq 0$, then the matrix $\tilde M={M\over t(T-1)}$ is feasible for the program defining $\vartheta'(G_t)$ with $\langle J, \tilde M\rangle =t$,  thus showing  the desired inequality $\vartheta'(G_t)\ge t$.

We now  show that $M\succeq 0$. We may assume  that all diagonal entries of $X$ are positive (else replace $X$ by its principal submatrix having only positive diagonal entries). Then, $D\succ 0$ and  define 
$M'= (D^{-1/2} \otimes I_t) M (D^{-1/2}\otimes I_t) = (T-1) I_{nt} -(I_n - D^{-1/2}XD^{-1/2}) \otimes (J_t-I_t)$.
Clearly, $M\succeq 0$ if and only if $M'\succeq 0$, which in turn is equivalent to checking that the eigenvalues of the matrix 
$Y=(I_n - D^{-1/2}XD^{-1/2}) \otimes (J_t-I_t)$ are at most $T-1$.
Let $0\le \lambda_1\le \ldots \le \lambda_n$ denote the eigenvalues of the positive semidefinite matrix $D^{-1/2}XD^{-1/2}$.
Then, the eigenvalues of $Y$ are 
$(1-\lambda_i)(t-1)$ and $(1-\lambda_i)(-1)$ for $i\in [n]$. Clearly, $(1-\lambda_i)(t-1)\le t-1\le T-1$ for all $i\in [n]$ and thus it suffices to show that $(1-\lambda_i)(-1) = \lambda_{i} - 1 \le T-1$ for all $i \in [n]$ or, equivalently, that $\lambda_n\le T$.
To this end, notice that 
since the matrix $D^{-1/2}XD^{-1/2}$ is nonnegative, by Perron-Frobenius it admits a nonnegative (unit) eigenvector $u$ for its largest eigenvalue $\lambda_n$. Define the  matrix $X'= D^{-1/2}XD^{-1/2}\circ uu^T\in \mathcal S^n$ (taking the entrywise product). Then, $X'\in \dnn^n$, $X'_{uv}=0$ if $\{u,v\}\in E(G)$, $\Tr(X')=\|u\|^2=1$, and
$\langle J,X'\rangle = u^T D^{-1/2}XD^{-1/2} u =\lambda_n$. As $X'$ is feasible for the program defining $\vartheta'(G)$, it follows that $\lambda_n\le \vartheta'(G)=T$.
\end{proof}

\begin{proposition}\label{lem:theta'equiv}
Let $G$ be a graph, let $t\ge 1$ be an integer,  and let $\mathcal K$ denote  the   cone $\dnn$ or $\cp$.  The following statements are equivalent.
\begin{itemize}
\item[\namedlabel{1-stab}{(i)}] There exists a matrix $X \in \mathcal{K}^{|V(G)|}$ satisfying $\floor{\langle J, X \rangle} = t$, $\Tr(X) = 1$ and $X_{uv} = 0$  for all $\{u,v\} \in E(G)$.
\item[\namedlabel{2-stab}{(ii)}] There exists a matrix $X \in \mathcal{K}^{|V(G)|t+1}$ satisfying the conditions (\ref{eq:cqalpha-1}), (\ref{eq:cqalpha-sum2a}), (\ref{eq:cqalpha-sum-2}), (\ref{eq:cqalpha-orth1}) and (\ref{eq:cqalpha-orth2}).
\item[\namedlabel{3-stab}{(iii)}] There exists a matrix $X \in \mathcal{K}^{|V(G)|t+1}$ satisfying the conditions (\ref{eq:cqalpha-1}), (\ref{eq:cqalpha-sum2a}), (\ref{eq:cqalpha-sum-2}) and (\ref{eq:cqalpha-orth1}).
\end{itemize}
\end{proposition}

\begin{proof} 
Set $n=|V(G)|$. 
We will show the implications   $\ref{1-stab} \Rightarrow\ref{3-stab} \Rightarrow$ $\ref{2-stab} \Rightarrow$ $\ref{1-stab}$. 
Notice that statement $\ref{1-stab}$ is equivalent to $\vartheta^{\mathcal K}(G) \ge t$.
First we consider  the case $\mathcal K=\dnn$.

$\ref{1-stab} \Rightarrow\ref{3-stab}$:  Assume first $t=1$.
If  $\ref{1-stab}$ holds with $t=1$, then by Remark \ref{remK} there exists  a matrix $X\in \dnn^n$ with $\Tr(X)=\langle J,X\rangle =1$ and thus $X_{uv}=0$ for all $u\ne v\in V(G)$.
Let  $x_u$ (for all $u\in V(G)$) be a Gram representation of $X$, set $x_0=\sum_{u\in V(G)}x_u$, and define the matrix $Y'\in \mathcal S^{n+1}$ with Gram representation $x_0,x_u$ ($u\in V(G)$). Then, $Y'$ satisfies $\ref{3-stab}$.

Assume now $t\ge 2$.
If $\ref{1-stab}$ holds, then $\vartheta'(G) \ge t$ and from Lemma \ref{lemthetap} we can conclude that $\vartheta'(G_t)\ge t$. Thus,
by Remark \ref{remK}, there exists a matrix $X\in \mathcal S^{nt}$ feasible for the program  (\ref{opt:theta-str}) defining $\vartheta'(G_t)$ with value 
$\langle J,X\rangle =t$.
Hence the matrix $Y=tX \in \dnn^{nt}$ satisfies $\langle J,Y\rangle =t^2$, 
$\Tr(Y)=t$, and $Y_{ui,vj}=0$ for all edges $\{(u,i),(v,j)\}$ of $G_t$. Moreover, after symmetrization by $\Sym(t)$, we can assume that $Y$ has the block-form (\ref{eqblock}),
where $A$ is a diagonal matrix and $B_{uv}=0$ for all edges $\{u,v\}$ of $G$. 
Then, $t=\Tr(Y)=t\Tr(A)=t \langle J,A\rangle$ and
$t^2=\langle J,Y\rangle = t\langle J,A\rangle + t(t-1)\langle J,B\rangle$, implying
$\Tr(A)=\langle J,A\rangle =\langle J,B\rangle =1$ (since $t\ge 2$).

Let $\{y_u^i: u \in V(G),i\in [t]\}$ be a Gram factorization of $Y$, i.e., 
 $Y_{ui,vj} = \langle y_u^{i}, y_v^j\rangle$  for all $ i,j \in [t]$ and $u,v \in V(G)$.
Fix $i_0\in [t]$ and  define the vector $y = \sum_{u \in V(G)} y_u^{i_0}$. 
Then, $\langle y,y \rangle = \sum_{v \in V(G)} \langle y,y_v^{i_0} \rangle = \sum_{u,v \in V(G)} \langle y_{u}^{i_{0}},y_v^{i_0} \rangle = \langle J,A \rangle = \Tr(A) = 1$ and, for any $j \in [t]\backslash\{i_0\}$, we have  $\sum_{v \in V(G)} \langle y,y_v^j \rangle = \sum_{u,v \in V(G)} \langle y_u^{i_0},y_v^j \rangle = \langle J,B \rangle = 1$. Define $Y'$ to be the Gram matrix of the vectors  $y,y_u^i$ (for $u \in V(G),i\in [t])$.
From the properties just explained, we see that $Y' \in\dnn^{nt+1}$ satisfies $\ref{3-stab}$.

\smallskip
$\ref{3-stab} \Rightarrow$ $\ref{2-stab}$:
Assume  that $Y'$  satisfies $\ref{3-stab}$; we construct a new matrix $Y$ satisfying $\ref{2-stab}$.
For this, consider the $(i,i)$-th diagonal block $Y'[ii]=(Y'_{ui,vi})_{u,v\in V(G)}$ of $Y'$. 
It suffices to show how to modify each $Y'[ii]$ in such a way that its off-diagonal entries become  zero. The idea is simple: just `move' the value of each off-diagonal entry $Y'_{ui,vi}$ to the diagonal entry $Y'_{ui,ui}$. Formally, for any $u\ne v\in V(G)$, define the matrix $F^{uv}\in \mathcal S^n$ with  entries $F^{uv}({uv})=F^{uv}({vu})=-1$, $F^{uv}({uu})=F^{uv}({vv})=1$ and all remaining entries zero. Then, $F^{uv}\succeq 0$.
 Moreover, for $i\in [t]$, define the matrix $F^{uv}_i\in \mathcal S^{nt+1}$ with $F^{uv}$ as its $(i,i)th $ diagonal block and all remaining entries equal to 0, so that 
 $F^{uv}_i\succeq 0$.
 Fix an arbitrary ordering of the vertices of $G$. 
 Define the new matrix
 \begin{equation}\label{eqYpY}
 Y= Y' + \sum_{i\in [t]} \sum_{u< v \in V(G)} Y'_{ui,vi}F^{uv}_i.\end{equation}
By construction, the sum of entries of the $(i,i)$-th diagonal block of $Y$ is equal to the sum of entries of the $(i,i)$-th diagonal block of $Y'$ and thus to 1. 
The matrix $Y$ is entrywise nonnegative and it is a sum of positive semidefinite matrices. It then follows that 
$Y$ satisfies $\ref{2-stab}$.

\smallskip
$\ref{2-stab} \Rightarrow$ $\ref{1-stab}$:
Let $Y$ be a matrix satisfying $\ref{2-stab}$. As $Y\succeq 0$, there exists   vectors  $y, y^u_i$ (for $u \in V(G), i \in [t]$)  forming a Gram representation of $Y$.
For  $i \in [t]$,  we have:
$\| y - \sum_{u \in V(G)} y_i^u \|^2 = Y_{0,0} - 2 \sum_{u \in V(G)} Y_{0,ui} + \sum_{u,v \in V(G)} Y_{ui,vi} = 0$ (using (\ref{eq:cqalpha-1}),(\ref{eq:cqalpha-sum2a}),(\ref{eq:cqalpha-sum-2})), which implies
that $y=\sum_{u\in V(G)}y^u_i$ for all $i \in [t]$. 
Define the vectors $x_u = \sum_{i\in [t]} y^u_i$ for all $u\in V(G)$ and let $X\in \mathcal S^{|V(G)|}$  denote their Gram matrix.
Then,  $X \succeq 0$, $\langle J,X \rangle = \|\sum_{u\in V(G)}\sum_{i=1}^t y^u_i\|^2= \|t y\|^2=t^2$,
and $\Tr(X)= \sum_{u \in V(G)} \|x_u\|^2 = \sum_{i,j \in [t]} \sum_{u \in V(G)} \langle y_i^u, y_j^u \rangle = \sum_{i \in [t]} \sum_{u \in V(G)} Y_{ui,ui} = t$. Moreover, $X_{uv} = \langle x_u,x_v\rangle = \sum_{i,j \in [t]} \langle y_i^u, y_j^v \rangle =\sum_{i,j \in [t]} Y_{ui,vj} \geq 0$ for any $u,v \in V(G)$, with  equality for $\{u,v\} \in E(G)$. Rescaling the matrix $X$ by $1/t$, we obtain a feasible solution for $\ref{1-stab}$.
This concludes the proof in the case $\mathcal K=\dnn$.

\medskip
We now consider the case  $\mathcal{K} = \cp$.  

$\ref{1-stab}$ $\Rightarrow$ $\ref{3-stab}$:
Let $X$ be a matrix that satisfies $\ref{1-stab}$. Applying Theorem \ref{theodeK}, 
 we obtain that $\alpha(G) \geq t$.
Let $S\subseteq V(G)$ be  a stable set of cardinality $t$. Say, $V(G)=[n]$ and $S=\{1,\ldots,t\}$. Define the vector $y\in \R^{nt+1}_{+}$ with block-form $y=(1,e_1,\ldots,e_t)$, where $e_1,\ldots,e_t$ are the first $t$ standard unit vectors in $\R^n$. Define the matrix $Y'=yy^T$ which, by construction, belongs to $\cp^{nt+1}$. It is easy to verify that $Y'$ satisfies $\ref{3-stab}$.

\smallskip
$\ref{3-stab}$ $\Rightarrow$ $\ref{2-stab}$:
We can mimic the above proof of this implication in  the case of the cone $\dnn$. The only thing to notice is that  the new matrix $Y$  in (\ref{eqYpY}) 
is  completely positive, which can proved by applying Theorem \ref{lem:2*2cp}. 
Indeed,  $Y'\in \cp$, each  term $Y'_{ui,vi}{F}^{uv}_i$ is a positive semidefinite matrix whose entries are all zero except  for a $2 \times 2$ principal submatrix,
and one gets a nonnegative matrix at each intermediate step of the summation. Hence, Theorem~\ref{lem:2*2cp} can be applied at every step and one can  conclude that $Y\in \cp$.

\smallskip
$\ref{2-stab}$ $\Rightarrow$ $\ref{1-stab}$:  
The proof is analogous to the above proof of this implication for $\dnn$.
 \end{proof}

\noindent
As an application, if   in Proposition \ref{def:cgqindep} we replace 
 the cone $\cgpsd$ by the cone $\dnn$ %(resp., by the cone $\cp$)
  in the definition of $\alpha_q(G)$ or of $\qindep(G)$, then we obtain the parameter $\lfloor \vartheta'(G)\rfloor$;  
  analogously, if we replace the cone $\cgpsd$ by the cone $\cp$ then we obtain  $\alpha(G)$.

\begin{corollary}\label{cor:theta'equiv}
For any graph $G$, the maximum integer $t$ for which there exists a matrix $X\in \mathcal K^{|V(G)|t+1}$ satisfying the conditions (\ref{eq:cqalpha-1}), (\ref{eq:cqalpha-sum2a}), (\ref{eq:cqalpha-sum-2}), (\ref{eq:cqalpha-orth1}) and (\ref{eq:cqalpha-orth2}) 
(or, equivalently, the conditions  (\ref{eq:cqalpha-1}), (\ref{eq:cqalpha-sum2a}), (\ref{eq:cqalpha-sum-2}) and (\ref{eq:cqalpha-orth1}))
is equal to the parameter $\lfloor \vartheta'(G)\rfloor$ when $\mathcal K=\dnn$ and it is equal to the stability number $\alpha(G)$ when $\mathcal K=\cp$.
\end{corollary}

\begin{proof}
This follows  by applying  Proposition \ref{lem:theta'equiv} combined with the definition of $\vartheta'$ in (\ref{opt:theta-str}) when $\mathcal K=\dnn$ and with Theorem \ref{theodeK} when $\mathcal K=\cp$.
 \end{proof}

\noindent
In turn this permits to derive  the following `sandwich inequalities' for the quantum analogues of the stability number.

\begin{corollary}\label{corsandqa}
For any graph $G$, 
$
\alpha(G) \leq \alpha_q(G) \leq \qindep(G) \leq \floor{\vartheta'(G)}.
$
\end{corollary}

\noindent
The  bound $\qindep(G) \leq \floor{\vartheta'(G)}$ was shown recently, with a different method, by Cubitt et al.~\cite{Cubitt:2013}.
The inequality $\alpha(G)\le \alpha_q(G)$ can be strict (see \cite{Roberson:2012}), but it is not known whether the other two inequalities can be strict.

Observe that, if one could prove that the two conditions $\ref{2-stab}$ and $\ref{3-stab}$ in Proposition~\ref{lem:theta'equiv} are  equivalent  also when setting $\mathcal{K} = \cgpsd$, then this would imply that equality $\alpha_q(G) = \qindep(G)$ holds.
This would work if we could show the analogue of Theorem~\ref{lem:2*2cp} when replacing the condition of being `completely positive' by the condition  of being `completely positive semidefinite',  since then the reasoning  used in the proof of Proposition~\ref{lem:theta'equiv} for the implication $\ref{3-stab}$ $\Rightarrow$ $\ref{2-stab}$ would extend to the case of $\cgpsd$.
However, the following example shows that Theorem \ref{lem:2*2cp} does not extend to the cone $\cgpsd$.

\begin{example}\label{ex2*2cgpsd}
Consider the matrix $L= M(\cos^2(\frac{4 \pi}{5}),\cos^2(\frac{2 \pi}{5}))$, which was presented in Section~\ref{seclink} as an example of a matrix which is completely psd but not completely positive. 
For $i\ne j\in [5]$,  let $F^{ij}\in \mathcal S^5_+$ be the matrix with all zero entries except $F^{ij}_{ii}=F^{ij}_{jj}=1$ and $F^{ij}_{ij}= F^{ij}_{ji}=-1$. Define the matrix $L' =L+\cos^2({\frac{2\pi}{5}})(F^{13}+F^{24}+F^{35}+F^{14}+F^{25})$. 
Then, $L'$ is not completely positive,  since its inner product with the Horn matrix is negative. Indeed,
$\langle H,L' \rangle = 
5(1+2\cos^2({\frac{2\pi}{5}})) -10 \cos^2({\frac{4\pi}{5}}) = 5(2-\sqrt 5)/2<0$.
As the support of $L'$ is equal to the 5-cycle, we can conclude using  Theorem~\ref{lem:oddcycle} that $L'$ is not completely positive semidefinite.

Thus, although one gets  nonnegative matrices at each step of the summation defining $L'$ starting from $L\in \cgpsd^5$,   the final matrix $L'$ does not belong to the cone $\cgpsd^5$. 
\end{example}

Finally, we relate the quantum stability number $\alpha_q(G)$  with the generalized theta number $\vartheta^{\cgpsd}(G)$, obtained when selecting the cone $\mathcal K=\cgpsd$ in the definition (\ref{eqthetas-K}).

\begin{proposition}\label{propsandwichalphaq}
For any graph $G$, we have: $\alpha_q(G) \le \floor{\vartheta^{\cgpsd}(G)} \le \floor{\vartheta^{\cl(\cgpsd)}(G)} \le \floor{\vartheta'(G)}.$
\end{proposition}

\begin{proof}
The equality  $\vartheta^{\dnn}(G) = \vartheta'(G)$ (from (\ref{eqlinktheta})) and the inclusions $\cgpsd \subseteq \cl(\cgpsd)\subseteq \dnn$ give  $\vartheta^{\cgpsd}(G) \le \vartheta^{\cl(\cgpsd)}(G) \le \vartheta'(G)$ and thus the two right most equalities.

We now show the inequality $\alpha_q(G)\le \floor{\vartheta^{\cgpsd}(G)}$.
For this, we revisit the proof of Proposition~\ref{lem:theta'equiv}.
First we observe that  the implication $\ref{2-stab}\Rightarrow \ref{1-stab}$  remains true in 
 Proposition \ref{lem:theta'equiv} if we  select the cone $\K=\cgpsd$. (Indeed, the same proof applies as in the case $\K=\dnn$, except that   $y,y^u_i$ are now psd matrices.)
By definition,  $\alpha_q(G)$ is the largest integer $t$ for which   Proposition~\ref{lem:theta'equiv} $\ref{2-stab}$ holds with $\K=\cgpsd$.
In turn, by the above,  this largest number  is at most the largest integer $t$ for which 
Proposition~\ref{lem:theta'equiv} $\ref{1-stab}$ holds with $\K=\cgpsd$, the latter being equal to $\floor{\vartheta^{\cgpsd}(G)}$. Thus 
$\alpha_q(G)\le \floor{\vartheta^{\cgpsd}(G)}$ holds.
\end{proof}

\noindent
We do not know whether $\vartheta^{\cgpsd}(G)$ also provides an upper bound for $\qindep(G)$, since
we cannot show that Proposition \ref{lem:theta'equiv} $\ref{3-stab}$ implies Proposition \ref{lem:theta'equiv} $\ref{1-stab}$ in the case $\mathcal K = \cgpsd$. The proof used for the case $\mathcal K = \dnn$ and $\cp$ indeed 
does not extend to the case $\K=\cgpsd$ since  Theorem \ref{lem:2*2cp} does not hold if we consider matrices in 
$\cgpsd$ (as shown in Example \ref{ex2*2cgpsd}).

\subsection{Conic reformulation for quantum chromatic numbers}\label{secconechi}

Analogously to what we did  for the quantum stability numbers, we can reformulate the two quantum variants $\chi_q(G)$  and $\qchrom(G)$ of the chromatic number as conic feasibility programs over the cone $\cgpsd$. The proof is omitted since it is easy and along the same lines as  for Proposition~\ref{def:cgqindep}.

\begin{proposition}
\label{def:cgqchrom}
For a graph $G$, $\chi_q(G)$ is equal to  the minimum integer $t$ for which there  exists a matrix $X \in \cgpsd^{|V(G)|t+1}$ satisfying   the following conditions:
\begin{align}
&X_{0,0} = 1, \label{eq:cqchrom-1} \tag{C1}\\
& \sum_{i \in [t]} X_{0,ui}= 1\quad \forall u\in V(G), \label{eq:cqchrom-sum-a} \tag{C3a}\\
&\sum_{i,j \in [t]} X_{ui,uj} =1 \quad \forall u\in V(G) \label{eq:cqchrom-sum3} \tag{C3b}.\\
&X_{ui,vi} = 0 \quad \forall i \in [t], \, \forall \{u,v\} \in E(G), \label{eq:cqchrom-orth1} \tag{O3}\\
&X_{ui,uj} = 0 \quad \forall i\neq j \in [t], \, \forall u \in V(G)\label{eq:cqchrom-orth2} \tag{O4}.
\end{align}
Moreover, the parameter $\qchrom(G)$ is equal to the minimum integer $t$ for which there  exists a matrix $X \in \cgpsd^{|V(G)|t+1}$ satisfying   
(\ref{eq:cqchrom-1}), (\ref{eq:cqchrom-sum-a}),  (\ref{eq:cqchrom-sum3}) and (\ref{eq:cqchrom-orth1}).
\end{proposition}

We now show an analogue of Proposition~\ref{lem:theta'equiv} for the chromatic parameters.
However, as pointed out in Remark \ref{remqcp} below, when chosing the cone $\K=\cp$, the statements $\ref{1-chrom}$ and $\ref{2-chrom}$ in Proposition \ref{lem:equithetaplus} are not equivalent.

\begin{proposition}\label{lem:equithetaplus}
Let $G$ be a graph,  $t\ge 1$ be an integer, and let $\mathcal{K}$ denote the cone  $\dnn$ or $\cp$. 
Consider the following three assertions.
\begin{itemize}
\item[\namedlabel{1-chrom}{(i)}] There exists a matrix $X \in \mathcal{K}^{|V(G)|}$ such that $\ceil{X_{uu}} = t$ for every $u \in V(G)$, $X_{uv} = 0$ for all $\{u,v\} \in E(G)$ and $X - J \succeq 0$.
\item[\namedlabel{2-chrom}{(ii)}] There exists a matrix $X \in \mathcal{K}^{|V(G)|t+1}$ satisfying the conditions (\ref{eq:cqchrom-1}), (\ref{eq:cqchrom-sum-a}), (\ref{eq:cqchrom-sum3}), (\ref{eq:cqchrom-orth1}) and (\ref{eq:cqchrom-orth2}).
\item[\namedlabel{3-chrom}{(iii)}] There exists a matrix $X \in \mathcal{K}^{|V(G)|t+1}$ satisfying the conditions (\ref{eq:cqchrom-1}), (\ref{eq:cqchrom-sum-a}), (\ref{eq:cqchrom-sum3}) and (\ref{eq:cqchrom-orth1}).
\end{itemize}
Then, $\ref{1-chrom}$ $\Longleftrightarrow$ $\ref{2-chrom}$ $\Longleftrightarrow$ $\ref{3-chrom}$ if $\mathcal K=\dnn$, and $\ref{3-chrom}$ $\Longleftrightarrow$ $\ref{2-chrom}$ $\Longrightarrow $ $\ref{1-chrom}$ if $\mathcal K=\cp$.
 \end{proposition}

\begin{proof}
Notice that statement $\ref{1-chrom}$ is equivalent to saying that $\Theta^{\mathcal K}(G) \le t$ holds.

Assume first  $\mathcal{K} = \dnn$. 
We show: ${ \ref{1-chrom}} \Rightarrow {\ref{3-chrom}} \Rightarrow { \ref{2-chrom}} \Rightarrow { \ref{1-chrom}}$.

${ \ref{1-chrom}} \Rightarrow { \ref{3-chrom}}$:
Let $X$ be a matrix that satisfies the conditions of $\ref{1-chrom}$.
By adding a nonnegative diagonal matrix to $X$ we can assume  that $X_{uu} = t$ for all $u \in V(G)$.
Set $X' = X-J\in \mathcal S^{|V(G)|}$.
Then, $X'\succeq 0$, $X'_{uu}=t-1$ for all $u\in V(G)$ and, for $u\ne v$,  $X'_{uv}= X_{uv}-1\ge -1$ with equality when $\{u,v\}\in E(G)$.
Moreover, $X'_{uv}\ge -(t-1)$ since $X'\succeq 0$ with diagonal entries equal to $t-1$. 

In the case  $t=1$, we have $X'=0$, hence $G$ is the empty graph and the all-ones matrix satisfies $\ref{3-chrom}$.
We now assume $t\ge 2$.  We define the matrices
$\tilde{A} =\frac{1}{t^2} X'$, $\tilde{B} = - \frac{1}{t^2(t-1)}X'$, $A = \tilde{A} + \frac{1}{t^2} J$ and $B = \tilde{B} + \frac{1}{t^2} J \in \mathcal S^{|V(G)|}$. We let $Y\in \mathcal S^{|V(G)|t}$ be the block-matrix as in (\ref{eqblock}) with $A$ as diagonal blocks and $B$ as off-diagonal blocks and $Y' =\left(\begin{matrix} 1 & \frac{1}{t} e^T\cr \frac{1}{t} e & Y\end{matrix}\right)$. 
We now show that $Y'\in \mathcal S^{|V(G)|t+1}$ satisfies $\ref{3-chrom}$.

By construction, (\ref{eq:cqchrom-1}), (\ref{eq:cqchrom-sum-a}) and  (\ref{eq:cqchrom-orth1}) hold.
(\ref{eq:cqchrom-sum3}) follows from the simple observation that $tA+t(t-1)B=J$ and thus $\sum_{i,j \in [t]} Y'_{ui,uj} = tA_{uu} + t(t-1)B_{uu} = 1$ for every $u \in V(G)$.
At last we argue that $Y' \in \dnn$.
Notice that $A,B\ge 0$ and thus $Y,Y'\ge 0$. 
Moreover, doing the Schur complement of $Y$ in $Y'$ w.r.t. its $(0,0)$-th entry (recall (\ref{eqSchur})), we obtain that $Y' \succeq 0$ if and only if $\tilde{Y} = Y - {\frac{1}{t^2}}J \succeq 0$. 
Now, $\tilde{Y}$ has the block structure of (\ref{eqblock}) with $\tilde{A}$ and $\tilde{B}$ as diagonal and off-diagonal blocks, respectively.
Moreover, $\tilde{A}+(t-1)\tilde{B}=0$ and $\tilde{A}-\tilde{B}= {\frac{1}{t(t-1)}}X'\succeq 0$ and thus, by Lemma \ref{lem:block}, we deduce that $\tilde Y\succeq 0$ and therefore $Y' \succeq 0$.

\smallskip
$ \ref{3-chrom} \Rightarrow \ref{2-chrom}$: 
Let $Y'$ be a feasible matrix for $\ref{3-chrom}$, we construct a new matrix $Y$ satisfying  $\ref{2-chrom}$.
As in the proof of the implication $\ref{3-chrom}\Rightarrow \ref{2-chrom}$ in Proposition~\ref{lem:theta'equiv}, it suffices to modify each $(u,u)$-th diagonal block of $Y'$ in such a way that all its off-diagonal entries become zero.
For this, for $i\ne j\in [t]$, consider the matrix $F^{ij}\in \mathcal S^t$ with entries $F^{ij}({ij})=F^{ij}({ji})=-1$, $F^{ij}({ii})=F^{ij}({jj})=1$, and all remaining entries equal to 0.
Moreover, for $u\in V(G)$, define the matrix $F^{ij}_u\in \mathcal S^{nt+1}$ with $F^{ij}$ as its $(u,u)$-th diagonal block and all remaining entries equal to 0, so that
 $F^{ij}_u\succeq 0$.
The new matrix 
\begin{equation}\label{eqYpYb}
Y=Y' + \sum_{u\in V(G)} \sum_{1\le i<j\le t} Y'_{ui,uj} F^{ij}_u,
\end{equation}
is entrywise nonnegative, positive semidefinite and satisfies all the conditions in  $ \ref{2-chrom}$.

\smallskip
${\ref{2-chrom}} \Rightarrow \ref{1-chrom}$:
Let $Y\in \dnn$ satisfy $\ref{2-chrom}$. 
Without loss of generality, we can assume that 
$Y$ has the block-form (\ref{eq:block}). Then, $\alpha=Y_{00}=1$ by (\ref{eq:cqchrom-1}),
$a={\frac{1}{t}}e$ by (\ref{eq:cqchrom-sum-a}),
$A_{uu}={\frac{1}{t}}$ for all $u\in V(G)$ by (\ref{eq:cqchrom-sum3}) together with (\ref{eq:cqchrom-orth2}),
$A_{uv}=0$ for $\{u,v\}\in E(G)$ by (\ref{eq:cqchrom-orth1}), and
$B_{uu}=0$ for $u\in V(G)$ by (\ref{eq:cqchrom-orth2}).
Let $Z\in \mathcal S^{|V(G)|t}$ denote the principal submatrix of $Y$ obtained by deleting its first row and column indexed by the index $0$, so that $Z$ has the block-form (\ref{eqblock}). Let $Z'$ denote the Schur complement of $Z$ in $Y$ w.r.t. its $(0,0)$-th entry (recall (\ref{eqSchur})). 
Using the fact that  $a=e/t$, we obtain that 
$Z'=Z-{\frac{1}{t^2}}J $. Moreover, $Y\succeq 0$ implies $Z'\succeq 0$. Now, $Z'$ has again the block-form (\ref{eqblock}) with diagonal blocks $A'=A-{\frac{1}{t^2}}J$ and with off-diagonal blocks $B'=B-{\frac{1}{t^2}}J$. Applying Lemma \ref{lem:block}, we deduce that $A'-B'\succeq 0$ and $A'+(t-1)B'\succeq 0$, which implies $A-B\succeq 0$ and  
$A+(t-1)B-{\frac{1}{t}}J\succeq 0$.
Now observe that $\Tr(A+(t-1)B-{\frac{1}{t}}J) = \Tr(A-{\frac{1}{t}}J)=0$ and that this implies $A+(t-1)B-{\frac{1}{t}}J=0$ as $A+(t-1)B-{\frac{1}{t}}J \succeq 0$. 

We can now construct a matrix $X\in \mathcal S^{|V(G)|}$ satisfying $\ref{1-chrom}$.
Namely, set $X=t^2A$. Thus, $X\in \dnn$, $X_{uu}=t$ for $u\in V(G)$, and $X_{uv}=0$ for $\{u,v\}\in E(G)$.
Moreover, $X-J\succeq 0$, $X-J = t^2A-J = t(t-1)(A-B)$ which follows
 from $A-B\succeq 0$ and  the identity $A+(t-1)B={\frac{1}{t}}J$.
This concludes the proof in the case $\mathcal K=\dnn$.

\medskip
We now consider the case  $\mathcal{K} = \cp$. 
The implication $\ref{2-chrom}\Rightarrow \ref{3-chrom}$ is clear. 

\smallskip
$ \ref{3-chrom} \Rightarrow \ref{2-chrom}$:
We can mimic the above proof of this implication in  the case of the $\dnn$ cone.
We only need to observe that the new matrix $Y$ in (\ref{eqYpYb})
is  completely positive. This is the case because
Theorem~\ref{lem:2*2cp} can be applied at every step of the summation, since one gets a nonnegative matrix at each step.

\smallskip
$\ref{2-chrom}$ $\Rightarrow$ $\ref{1-chrom}$: Again we can mimic   the above proof of this implication in the case of $\dnn$. Indeed,  we can assume that there exists a matrix $Y\in \cp^{|V(G)|t+1}$ satisfying $\ref{2-chrom}$ and with block-form (\ref{eq:block}), where $A,B$ satisfy the identity:
$A+(t-1)B={\frac{1}{t}}J$. 
Then, the matrix $X=t^2A$ belongs to $\cp^{|V(G)|}$ and satisfies $\ref{1-chrom}$.
 \end{proof}

\begin{corollary}\label{corqchi}
For any graph $G$, the minimum integer $t$ for which there exists a matrix  $X\in \mathcal K^{|V(G)|t+1}$ satisfying the conditions 
(\ref{eq:cqchrom-1}), (\ref{eq:cqchrom-sum-a}), (\ref{eq:cqchrom-sum3}), (\ref{eq:cqchrom-orth1}) and (\ref{eq:cqchrom-orth2}) (or, equivalently,  the conditions
(\ref{eq:cqchrom-1}),(\ref{eq:cqchrom-sum-a}),(\ref{eq:cqchrom-sum3}) and (\ref{eq:cqchrom-orth1})) is equal to the parameter $\lceil \vartheta^+(\overline G)\rceil$ when $\mathcal K=\dnn$ and it is equal to the chromatic number $\chi(G)$ when $\mathcal K=\cp$.
\end{corollary}

\begin{proof}
In the case $\mathcal K=\dnn$, the result follows using Proposition \ref{lem:equithetaplus} combined with the definition of $\vartheta^+(G)$ from (\ref{opt:theta-str}).

Consider now the case $\mathcal K=\cp$. In view of Proposition \ref{lem:equithetaplus}, we know that the two conditions $\ref{2-chrom}$ and $\ref{3-chrom}$ are equivalent. Let $t$ denote the minimum integer for which the condition $\ref{2-chrom}$ of Proposition \ref{lem:equithetaplus} holds; we show that $\chi(G)=t$.
First, we show that $\chi(G)\le t$. For this, consider a matrix $Y\in \cp^{|V(G)|t+1}$ satisfying $\ref{2-chrom}$ which has block-form (\ref{eq:block}) and let $Z$ be its principal submatrix obtained by deleting its row and column indexed by $0$. Then, $Z\in \cp^{|V(G)|t}$ 
and let $A$ and $B$ denote its diagonal and off-diagonal blocks, respectively. 
As in the proof of implication $\ref{2-chrom}$ $\Rightarrow$ $\ref{1-chrom}$ in Proposition \ref{lem:equithetaplus}, we can deduce that $A-B \succeq 0$, $A + (t-1)B = \frac{1}{t}J $  and that $\Tr(A) = |V(G)|/t$.
This then implies that $\Tr(Z) =|V(G)|$ and $\langle J,Z\rangle =|V(G)|^2$.
Now we use the result of Theorem \ref{theodeK} for computing the value of $\alpha(G\Box K_t)$. For this, set $Z'={\frac{1}{|V(G)|}}Z\in \cp^{|V(G)|t}$.
We see that $Z'$ satisfies the conditions of the program (\ref{def:stab}) applied to the graph $G\Box K_t$. Indeed the orthogonality conditions 
(\ref{eq:cqchrom-orth1}) and (\ref{eq:cqchrom-orth2}) correspond exactly to the edges of $G\Box K_t$.
Therefore, we can deduce that $\alpha(G\Box K_t)\ge |V(G)|$. As the reverse inequality also holds (since $G \Box K_{t}$ can be covered by $|V(G)|$ cliques $K_t$), we have $\alpha(G\Box K_t)=|V(G)|$. Using the reduction of Chv\'atal in Theorem \ref{theoChvatal}, we can conclude that 
$\chi(G)\le t$. 

We now prove the reverse inequality: $t\le \chi(G)=:s$.  It is easy to see that $G\Box K_s$ can be properly colored with $s=\chi(G)$ colors.
Therefore, $\chi(G\Box K_s)=s$ holds.
We  construct a matrix $Y\in \cp^{|V(G)|s+1}$ satisfying the conditions of $\ref{2-chrom}$, which will imply $ t\le s$ and thus conclude the proof.
For this,  select $s$ subsets $S_1,\ldots,S_s\subseteq V(G\Box K_s)$ which  are stable sets in  $G\Box K_s$ and partition the vertex set of $G\Box K_s$.
 For $k\in [s]$, let $x^k\in \R^{|V(G)|s}$ denote the incidence vector of $S_k$ and set $y^k=(1, x^k)\in \R^{|V(G)|s+1}$. Finally, define the matrix
$Y=\frac{1}{s} \sum_{k=1}^s y^k(y^k)^T$. By construction,  $Y \in \mathcal \cp^{|V(G)|s+1}$ and  $Y$ satisfies conditions (\ref{eq:cqchrom-orth1}) and (\ref{eq:cqchrom-orth2}). 
Moreover $Y_{0,0}=1$, $Y_{0,ui} = Y_{ui,ui} =\frac{1}{s}$ for every $u \in V(G)$ and $i \in [s]$ and thus $Y$ also satisfies (\ref{eq:cqchrom-1}), (\ref{eq:cqchrom-sum-a}) and (\ref{eq:cqchrom-sum3}). Hence $Y$ is feasible for $\ref{2-chrom}$.
This concludes the proof.
 \end{proof}

\noindent
As an application we obtain the following `sandwich' inequalities for the quantum variants of the chromatic number.

\begin{corollary}\label{corsandqc}
For any graph $G$,
$
\lceil {\vartheta}^+(\overline G)\rceil \le \qchrom(G)\le \chi_q(G)\le  \chi(G).
$
\end{corollary}

\noindent
The inequality 
$ \lceil {\vartheta}^+(\overline G)\rceil \le \qchrom(G)$ was  shown recently  in \cite{Briet:2013}. 
Moreover, there is a graph for which $\chi_{q}(G)<\chi(G)$ \cite{Cameron:2007} and one for which $\lceil {\vartheta}^+(\overline G)\rceil < \chi_{q}(G)$ \cite{Cubitt:2013}.
%In the above chain of inequalities, only the right most one is known to be 
%strict for some graphs \cite{Cameron:2007}. 
%Moreover, there is a graph for which $\lceil {\vartheta}^+(\overline G)\rceil < \chi_{q}(G)$ \cite{Cubitt:2013}.
Note also that the quantum chromatic numbers are not upper bounded by the fractional chromatic number. For instance, for the 5-cycle $C_{n}$, $\chi_f(C_{5})=5/2$ while $\chi_q(C_{5})=3$. To see that $\chi_q(C_5)=3$ we use the result of \cite{Cameron:2007} showing that $\chi_q(G)\le 2$ if and only if $G$ is  bipartite.

\medskip
We further observe that, in Proposition \ref{lem:equithetaplus}, the implication $\ref{1-chrom}$ $\Rightarrow$ $\ref{2-chrom}$ does not hold when selecting the cone $\mathcal K=\cp$.

\begin{remark}\label{remqcp}
As we just saw in Corollary \ref{corqchi}, the 
 smallest integer $t$ for which there exists a matrix $X\in \cp^{|V(G)|t+1}$ satisfying Proposition \ref{lem:equithetaplus} $\ref{2-chrom}$ is equal to the chromatic number $\chi(G)$.
 On the other hand, as a direct application of Theorem~\ref{theoDR}, we see that the smallest integer $t$ for which there exists a matrix $X\in \cp^{|V(G)|}$ satisfying Proposition \ref{lem:equithetaplus} $\ref{1-chrom}$ is equal to $\lceil \chi_f(G)\rceil$, where $\chi_f(G)$ is the fractional chromatic number of $G$.
 The inequality $\lceil \chi_f(G)\rceil \le \chi(G)$ is consistent with the inequality $t\le s$ corresponding to the implication $\ref{2-chrom}$ $\Rightarrow$ $\ref{1-chrom}$ in 
 Proposition \ref{lem:equithetaplus}.
 
Moreover, the parameters $\lceil \chi_f(G)\rceil$ and $\chi(G)$ can differ significantly.
For $n\ge 2r$, consider the Kneser graph $K(n,r)$,  whose vertices are the subsets of size $r$ of $[n]$ and where two vertices are adjacent if the sets are disjoint.
 Then,    $\chi_f(K(n,r)) = \frac{n}{r}$~\cite{Lovasz:1979} and $\chi(K(n,r))= n-2r+2$ \cite{Lovasz:1978}. 
 This shows that the  implication $\ref{1-chrom}$ $\Rightarrow$ $\ref{2-chrom}$ does not hold in Proposition \ref{lem:equithetaplus} in the case $\mathcal K=\cp.$
\end{remark}

We conclude with a comparison of the quantum chromatic numbers with the generalized theta number $\Theta^{\cgpsd}(G)$, obtained by selecting the cone $\mathcal K=\cgpsd$ in the definition (\ref{eqthetas-K}).

\begin{proposition}\label{propsandwichchiq}
For any graph $G$, $\ceil{\vartheta^+(\overline{G})}\le \ceil {\Theta^{\cl(\cgpsd)}(G)} \le \ceil{\Theta^{\cgpsd}(G)} \le \qchrom(G) \le \chi_q(G).$
\end{proposition}

\begin{proof}
Combining the equality  $\Theta^\dnn(G) = \vartheta^+(\overline{G})$ (from (\ref{eqlinktheta})) with the inclusions $\cgpsd\subseteq \cl(\cgpsd)\subseteq \dnn$, we obtain  the two left most inequalities.

We now show the inequality $\ceil{\Theta^{\cgpsd}(G)} \le \qchrom(G)$.
For this, we use the fact that  $\ceil{\Theta^{\cgpsd}(G)}$ is the minimum integer $t$ for which  Proposition \ref{lem:equithetaplus} $\ref{1-chrom}$ holds when selecting $\mathcal{K}=\cgpsd$, and that 
 $\qchrom(G)$ is by definition the minimum integer $t$ for which  Proposition \ref{lem:equithetaplus} $\ref{3-chrom}$ holds with $\mathcal{K}=\cgpsd$. Therefore,  in order to prove  that $\ceil{\Theta^{\cgpsd}(G)} \le \qchrom(G)$ holds, it suffices to  show that Proposition \ref{lem:equithetaplus} $\ref{3-chrom}$ implies Proposition \ref{lem:equithetaplus} $\ref{1-chrom}$ also in the case $\K=\cgpsd$. This is what we do next.

Let $Y \in \cgpsd$ satisfy Proposition \ref{lem:equithetaplus} $\ref{3-chrom}$ with $\K=\cgpsd$. Again we may assume without loss of generality that $Y$ has the block-form (\ref{eq:block}). First we observe that we can use the initial part of the proof $\ref{2-chrom} \Rightarrow$ $\ref{1-chrom}$ to show that 
$A + (t-1)B - \frac{1}{t}J = 0$. The key observation is that condition (\ref{eq:cqchrom-sum3}) still implies that $\Tr(A + (t-1)B - \frac{1}{t}J)=0$.
Next, following the proof of $\ref{2-chrom} \Rightarrow$ $\ref{1-chrom}$, we consider the matrix $X = t^2 A$. Then $X \in \cgpsd$, $X_{uv} = 0$ for every $\{u,v\} \in E(G)$ and $X-J \succeq 0$. Since we started with a solution $Y$ of $\ref{3-chrom}$ (instead of a solution for $\ref{2-chrom}$), we can only derive that $X_{uu} \le t$ for any $u \in V(G)$. We now build a solution $X'$ by adding to $X$ a diagonal matrix $D$ with entries $D_{uu} = t-X_{uu} \geq 0$ for any $u \in V(G)$. Hence $X' \in \cgpsd$ and satisfies all the conditions of $\ref{1-chrom}$. This concludes the proof.
 \end{proof}

\subsection{Approximating the quantum graph parameters} \label{secPsi}

In this section we show how one can use the convex sets  $\Knceps$ introduced earlier in Section \ref{secapprsets} to define parameters that approximate the quantum graph parameters. 
We give the details only for the quantum chromatic number $\chi_q(G)$, but the same reasoning can be extended to the other parameters $\qchrom(G)$, $\alpha_q(G)$ and $\qindep(G)$. 

The construction will go as follows. In a first step we  reformulate $\chi_q(G)$ as a single `aggregated' minimization program over an affine section of the cone $\cgpsd$. When replacing the cone $\cgpsd$ by its closure $\cl(\cgpsd)$ we get the parameter $\wtilde \chi_q(G)$, satisfying $\chi(G)\ge \wtilde \chi_q(G)$. The second step will consist of writing the dual of this aggregated conic program over the cone $\cl(\cgpsd)$, which is thus a  maximization program over the dual cone $\cgpsd^*$, and show that strong duality holds.
Finally we  define the new graph parameters $\Psi_\epsilon(G)$ by replacing in this dual conic  program the cone $\cgpsd^*$ by  the convex sets $\Knceps$.

\medskip
We start with the formulation of $\chi_q(G)$ from Proposition \ref{def:cgqchrom}.
For convenience, we introduce the matrix $A^t_u \in \mathcal S^{nt+1}$ (for  $u\in V(G)$, $t\in [n]$), with entries 
$A_u^t(0,0) = A_u^t (ui,uj) = 1$ $\forall i,j \in [t]$, $A_u^t (0,ui) =A_u^t (ui,0) = -1$ $\forall i \in [t]$ and zero elsewhere, 
and we  set $A^t=\sum_{u\in V(G)}A^t_u$. Observe that each matrix $A^t_u$ is positive semidefinite (with rank 1).
These matrices are useful to formulate the constraints defining $\chi_q(G)$. 
Indeed, if (\ref{eq:cqchrom-1}) holds then the two conditions (\ref{eq:cqchrom-sum-a}), (\ref{eq:cqchrom-sum3}) are equivalent to $\langle A^t,X\rangle =0$. 
Therefore, by Proposition \ref{def:cgqchrom}, $\chi_q(G)$ is equal to the smallest $t\in \N$ for which there exists $X\in \cgpsd^{nt+1}$ satisfying the conditions (\ref{eq:cqchrom-1}), (\ref{eq:cqchrom-orth1}), (\ref{eq:cqchrom-orth2}) and $\langle A^t, X\rangle =0$.
We can now reformulate $\chi_q(G)$ as the optimal value of a single conic optimization program over the cone $\cgpsd$.

\begin{proposition}\label{propoptchiq}
Let $G$ be a graph and set $n = |V(G)|$. The quantum chromatic number  $\chi_q(G)$ is equal to the optimal value of the following program: 
\begin{equation}\label{optchiq}
\begin{array}{ll}
\min \ \sum_{t \in [n]} t X_{0,0}^t \quad \text{\rm s.t.} & X^t \in \cgpsd^{nt+1} \quad \forall t \in [n],\\
& \sum_{t \in [n]}  X_{0,0}^t = 1,\ \sum_{t\in [n]}  \langle A^t,X^t\rangle =0,\\ 
& X_{ui,vi}^t = 0 \quad \forall i\in [t], \, \forall \{u,v\} \in E(G), \; \forall t \in [n],\\
& X_{ui,uj}^t = 0 \quad \forall i\neq j\in [t], \, \forall u \in V(G), \; \forall t \in [n].
\end{array}
\end{equation}
\end{proposition}

\begin{proof}
Set $t=\chi_q(G)$ and let  $\mu$ denote the optimal value of the program (\ref{optchiq}).

Let $(t,X)$ be a solution for the program from Proposition \ref{def:cgqchrom} defining $\chi_q(G)$. We obtain a solution $X^1,\ldots,X^n$ to the program (\ref{optchiq}) by setting $X^t=X$ and $X^i=0$ if $i\in[n]\setminus \{ t\}$. This shows that $\mu \le t$.

Conversely, let $X^1,\ldots,X^n$ be a feasible solution for the program (\ref{optchiq}) and let $s$  be the minimum $i\in [n]$ such that $X_{0,0}^i \neq 0$. Then, the matrix  $X=X^s/X^s_{0,0}$
 is feasible for the program in Proposition \ref{def:cgqchrom}. Moreover we have: $t\le s  =s\sum_{i\in [n]}X^i_{0,0}= s \sum_{i \geq s} X_{0,0}^{i} \leq 
 \sum_{i \geq s} i X_{0,0}^{i} =  \sum_{i\in [n]} i X_{0,0}^{i}=\mu$. This shows that $t\le \mu$ and thus equality $\chi_q(G)=\mu$ holds.
 This also shows that program (\ref{optchiq}) indeed has an optimal solution, thus justifying writing `min' rather than `inf' in (\ref{optchiq}).
 \end{proof}

As the problem of deciding whether $\chi_q(G)\le 3$ is NP-hard \cite{Ji:2013}, it follows that linear optimization over affine sections of the completely positive semidefinite cone is an NP-hard problem.

\medskip
It is convenient to  rewrite  program (\ref{optchiq})  in a more compact way. For this set $N=\sum_{t=1}^n (nt+1)$, where $n=|V(G)|$, and  define the matrix  $A=\oplus_{t=1}^n A^t\in \mathcal S^N$.
Let  $E^t_{0,ui}$, $E^t_{ui,vj}$ denote the elementary matrices in $\mathcal S^{nt+1}$ and  let $\widetilde E^t_{0,ui}, \widetilde E^t_{ui,vj}$ denote their extensions to $\mathcal S^N$ obtained by adding zero entries.
Moreover, set  $C=\oplus_{t=1}^n t E^t_{0,0}$ and 
$B=\oplus_{t=1}^n E^t_{0,0}\in \mathcal S^N$.
Then we can  rewrite the program  (\ref{optchiq}) as follows:
\begin{equation}\label{eqchiqprimal}
\begin{array}{ll}
\chi_q(G)=\min\  \langle C,X\rangle \ \text{ s.t. }  & X\in \cgpsd^N,\ \langle B,X\rangle =1,\ \langle A,X\rangle =0,\ \\
& \langle \widetilde E^t_{ui,vi},X\rangle =0 \quad \forall i\in [t],\ \forall  \{u,v\}\in E(G),\ \forall t\in [n],\\
&  \langle \widetilde E^t_{ui,uj},X\rangle =0 \quad \forall i\ne j\in [t],\ \forall u\in V(G),\ \forall t\in [n].
\end{array}
\end{equation}

If we replace the cone $\cgpsd$ by its closure $\cl(\cgpsd)$ in the program (\ref{eqchiqprimal}), then  its optimal value is equal to  $\wtilde \chi_q(G)$ and we have: $\wtilde \chi_q(G) \le \chi_q(G).$ 
Note that it is not clear whether these two parameters coincide.
 This is 
because  the matrix $A$ is psd so that it belongs to the dual cone $\cgpsd^*$,  and thus  the constraint $\langle A,X\rangle=0$ implies that any feasible solution $X$ of (\ref{eqchiqprimal}) lies on the border of the cone 
$\cgpsd$. 
On the other hand, it is easy to verify that the result of Proposition \ref{propoptchiq} (and its proof) extend to the case when the cone $\cgpsd$ is replaced by its closure $\cl(\cgpsd)$. Hence, $\wtilde \chi_q(G)$ can be equivalently defined by using the program from Proposition \ref{def:cgqchrom} after replacing the cone $\cgpsd$ by its closure $\cl(\cgpsd)$.
Using this, Corollary \ref{corqchi} and the fact that $\cgpsd \subseteq \cl(\cgpsd) \subseteq \dnn$, we have the following inequalities:
\begin{equation}\label{eqtildechiqtheta}
\lceil\vartheta^+(\overline G)\rceil \le \wtilde \chi_q(G) \le \chi_q(G).
\end{equation} 

The dual program of (\ref{eqchiqprimal}) reads:
\begin{equation}\label{eqchiqdual}
\lambda_q(G):=\sup \ \lambda \ \text{ s.t. }  M=C-\lambda B -\mu A -\sum y^t_{u,v,i} \widetilde E^t_{ui,vi} -\sum z^t_{u,i,j} \widetilde E^t_{ui,uj} \in \cgpsd^{N*},
\end{equation}
where the variables are $\lambda,\mu, y^t_{u,v,i}$ and $z^t_{u,i,j}$, the first summation is over $t\in [n]$, $i\in [t]$ and 
$\{u,v\}\in E(G)$, and the second summation is over $t\in [n]$, $i\ne j\in [t]$ and $u\in V(G)$.
By weak duality,  we obtain  the inequality: $\lambda_q(G)\le \wtilde \chi_q(G)\le \chi_q(G)$. 

Moreover,  the program (\ref{eqchiqdual}) is strictly feasible, hence there is no duality gap and the optimal value of (\ref{eqchiqdual}) is equal to $\wtilde \chi_q(G)$; that is, 
$\lambda_q(G)=\wtilde \chi_q(G)\le \chi_q(G)$. 
To see that (\ref{eqchiqdual}) is strictly feasible, define the matrix  $M^t= (t+n^2) E^t_{0,0} + A^t -\sum_{u\in V(G)} \sum_{i\ne j\in [t]} E^t_{ui,uj}$ and set $M=\oplus_{t=1}^n M^t$.
Then, $M$ is feasible for the program (\ref{eqchiqdual}). Moreover, $M$ lies in the interior of $\cgpsd^*$ since $M \succ 0$, as $M^t\succ 0$ for all $t$. (Indeed note that  the entries are $M^t_{0,0}= n+t+n^2$,
$M^t_{0,ui}=-1$, $M^t_{ui,ui}=1$ and all other entries are zero, and take a Schur complement to see that $M^t\succ 0$).

\medskip
We  now introduce the new parameter $\Psi_\epsilon(G)$, which is obtained by  replacing in the program (\ref{eqchiqdual}) the cone $\cgpsd^*$ by the convex set $\Knceps$.

\begin{definition}\label{def:psi}
For $\epsilon\ge 0$,  
let $\Psi_\epsilon(G)$ denote the optimal value of the program:
\begin{equation}\label{eqPsieps}
\sup \ \lambda \ \text{\rm  s.t. }  M=C-\lambda B -\mu A -\sum y^t_{u,v,i} \widetilde E^t_{ui,vi} -\sum z^t_{u,i,j} \widetilde E^t_{ui,uj} \in \Knceps.
\end{equation}
\end{definition}

First we relate the parameter $\Psi_\epsilon(G)$ to the classical theta number.
\begin{lemma}\label{lem:Psi-epsilon}
For $\epsilon \ge 0$, we have: $\lceil \vartheta^+(\overline G)\rceil \le \Psi_\epsilon(G)$, with equality if $\epsilon=0$.
\end{lemma}

\begin{proof}
By Lemma \ref{lemBall}, we have the inclusion  $\dnn^*\subseteq \Knceps$, with equality if $\epsilon=0$. Hence the lemma will follow if we can show that the optimal value of the program (\ref{eqPsieps}) is equal to $\lceil \vartheta^+(\overline G)\rceil$ when we replace the set $\Knceps$ by its subset $\dnn^*$.

In other words, let us consider the program (\ref{eqchiqdual}) where we replace the cone $\cgpsd^*$ by the cone $\dnn^*$.
Using the same argument as above, we can conclude that its optimal value is equal to the optimal value of the program (\ref{eqchiqprimal}) where we replace the cone $\cgpsd$ by the cone $\dnn$ (strong duality holds and use the fact that  the cone $\dnn$ is closed).
Next, observe that this latter value (which is equal to the optimal value of the program (\ref{optchiq}) when we replace $\cgpsd$ by $\dnn$) is equal to $\lceil \vartheta^+(\overline G)\rceil$.
This can be seen by combining Proposition~\ref{lem:equithetaplus} together with the fact that the result of Proposition \ref{propoptchiq} (and its proof) extends to the case when we replace the cone $\cgpsd$ by the cone $\dnn$.
\end{proof}

As the sets $\Knceps$ aim to approximate the dual cone $\cgpsd^*$, the parameters $\Psi_\epsilon(G)$ aim to approximate the quantum coloring number $\chi_q(G)$. 
However, as there is no apparent inclusion relationship between $\cgpsd^{*}$ and $\Knceps$, we do not know the exact relationship  between  $\Psi_\epsilon(G)$ and $\chi_q(G)$. Moreover, as the cone $\cgpsd$ is not known to be closed, there is a possible gap between the two parameters $\chi_q(G)$ and $\wtilde \chi_q(G)$. Nevertheless, what we can claim is the following relationship under Connes' embedding conjecture.

\begin{lemma}
Assume that  Connes' embedding conjecture holds.
 Then,  $\wtilde \chi_q(G) \leq \inf_{\epsilon>0} \Psi_\epsilon(G)$.
 \end{lemma}
 
 \begin{proof}
 If Connes' conjecture holds then  $\cgpsd^* \subseteq \Knceps$  for any $\epsilon>0$ (Lemma \ref{lembasic}). The result now follows using the definition of $\Psi_\epsilon(G)$ and the definition of $\wtilde \chi_q(G)$ as the optimal value of (\ref{eqchiqdual}).
 \end{proof}

Finally, let us observe that the parameter $\Psi_\epsilon(G)$ can be obtained as the limit of a sequence of semidefinite programs.
For this, recall that $M$ lies in $\Knceps$ if the polynomial $p_M+\epsilon$ admits a decomposition of the form
$p_M+ \epsilon= g+h$, where   $g= \sum_{j=1}^{m_0} f_jf_j^* +\sum_{i=1}^n \sum_{j_i=1}^{m_i} g_{j_i} (1-X_i^2) g_{j_i}^*$
for some $f_j,g_{j_i}\in \ncP$ and $m_0,m_i\in \N$, and $h$ is a sum of commutators.
Now, fixing an integer $k$ and restricting to those decompositions of $p_M+\epsilon$ where all terms $f_jf_j^*$ and $g_{j_i} (1-X_i^2) g_{j_i}^*$ have degree at most $2k$, we get a parameter $\Psi_\epsilon^k(G)$ which can be computed via a semidefinite program (see, e.g., \cite{Burgdorf:2011} for details).
Moreover, $\Psi^k_\epsilon(G)$ tends to $\Psi_\epsilon(G)$ as $k$ goes to infinity.

\section{Concluding remarks}\label{secfinal}

We have introduced  the cone $\cgpsd$ of completely positive semidefinite matrices and studied some basic properties. 
However, the structure of this cone remains largely unknown. The first fundamental  open question is to settle whether the  cone $\cgpsd$ is closed.
A closely related open question is  whether the existence of a Gram representation by  {\em infinite} psd matrices in $\mathcal S^\N$ implies the existence of another Gram representation by {\em finite} psd matrices, the answer is positive if $\cgpsd$ is closed (in view of Theorem \ref{theoclosed}). 
A first step has been made recently in  \cite{BLP}, where it is shown that  the closure of the cone $\cgpsd$ consists of the  matrices admitting a Gram representation by positive elements in a specific finite von Neumann algebra $\mathcal M$. 
This algebra  $\mathcal M$ is constructed as the tracial ultraproduct of the usual matrix algebras $\R^{k\times k}$ ($k\ge 1$) and represents a natural extension of the algebra $\mathcal S^\N_+$.

The closedness question for $\cgpsd$ is quite similar in spirit to several open problems in the quantum information literature (see e.g. \cite{Leung:2013,Paulsen:2013}). For instance, it is not known  whether the set  of quantum bipartite correlations is closed (see \cite{Wehner:2008}). The recent work \cite{MR:2014} shows that the set of quantum bipartite correlations can be obtained as the projection of an affine section of the cone $\cgpsd$ and that  it would be  closed if  $\cgpsd$ is closed.

\medskip
If in   the definition of the quantum chromatic number $\chi_q(G)$  from Definition~\ref{def:gqchrom}, 
  instead of requiring that $\rho,\rho_u^i$ lie in 
$\mathcal S^d_+$ (for some $d\ge 1$), we require that $\rho,\rho_u^i$ lie in $\mathcal S^\N_+$, we obtain a (possibly different) parameter that we denote by $\chi_{q}^\infty(G)$.
Equivalently,  $\chi^\infty_q(G)$ can be formulated as linear optimization over an affine section of the cone $\mathcal{CS}_{\infty +}$
(the analogue of the fact that $\chi_q(G)$ can be formulated as linear optimization over an affine section of  $\cgpsd$).
Hence,   $\chi_q^\infty(G)\le \chi_q(G)$, with   equality  if $\cgpsd=\mathcal{CS}_{\infty +}$.
Moreover, as $\mathcal{CS}_{\infty +}\subseteq \cl(\cgpsd)$ (by Theorem \ref{theoclosed}), 
we also have that $\wtilde \chi_q(G)\le \chi_q^\infty(G)$. Hence we have the possible variations of the quantum chromatic number:
\begin{equation}\label{eqchiqs}
\wtilde \chi_q(G)\le \chi_q^\infty(G)\le \chi_q(G),
\end{equation}
with equality throughout if the cone $\cgpsd$ is closed.

 Moreover,  observe that if in the definition of $\chi_q(G)$ we would require that $\rho,\rho_u^i$ are positive compact operators on a Hilbert space $H$ and we rewrite the orthogonality conditions as $\rho^i_u\rho^i_v=0$ (for $\{u,v\}\in E(G), i\in [t]$) and $\rho_u^i\rho_u^j=0$ (for $i\ne j\in [t]$, $u\in V(G)$),  then we would get again the parameter $\chi_q^\infty(G)$.
Indeed,  by the first Hilbert-Schmidt theorem (see e.g. \cite[Thm 6.2.3]{EMT}), the Hilbert space $H$ can be decomposed as $H=\ker \rho \oplus H'$, where $H'$ is the closure of the image of $ \rho$ and admits an orthonormal base $\{f_k: k\in \N\}$ consisting of the eigenvectors of $\rho$. 
Let $\rho',\rho'^i_u$ denote the restrictions of $\rho,\rho^i_u$ to $H'$. Then, $\rho'\ne 0$ and $\rho',\rho'^i_u$ are positive operators on $H'$. Moreover, the operators $\rho'^i_u$  satisfy the same orthogonality conditions as the operators $\rho^i_u$ (since  $\ker \rho\subseteq \ker \rho_u^i$ for all $u,i$, which follows from positivity and the fact that $\rho=\sum_i\rho_u^i$ for all $u$).
Finally,  using the base $\{f_k: k\in \N\}$ of $H'$, the operators  $\rho',\rho'^i_u$ can be identified with matrices in $\mathcal S^\N_+$.

As we already observed in the preceding section, 
we do not know whether we can replace the cone $\cgpsd$ by its closure, for instance in Proposition  \ref{propoptchiq}. 
Denoting by $\mathcal A$ the affine space defined by the affine conditions in program (\ref{optchiq}), 
$\chi_q(G)$ is the minimum value of the objective function taken over $\cgpsd \cap \mathcal A$, which in turn  is equal to the minimum value taken over the closure of $\cgpsd \cap \mathcal A$.
Clearly, $\text{\rm cl}(\cgpsd \cap \mathcal A) \subseteq \mathcal A \cap \text{\rm cl}(\cgpsd)$. However we cannot prove that equality holds. If we could prove equality then this would imply  that equality holds throughout in (\ref{eqchiqs}).

\medskip
We have studied quantum analogues of several classical graph parameters. 
In particular, we have extended the known lower bound $\chi(G)\ge \vartheta^+(\overline G)$ to the quantum setting. We showed that $\chi_q(G)\ge \Theta^{\cgpsd}(G)$ and studied analogous relationships between the other quantum graph parameters and the various theta numbers.
As a step towards further approximations for the quantum chromatic number, we have introduced parameters $\Psi_\epsilon(G)$ defined by replacing the dual cone $\cgpsd^*$ with the convex sets $\Knceps$ in the dual program of $\wtilde \chi_q(G)$, where $\wtilde \chi_q(G) \le \chi_q(G)$.
However, the exact relationship between $\Psi_\epsilon(G)$ and $\wtilde \chi_q(G)$ is unknown and only if Connes' embedding conjecture holds we can claim that  $\wtilde \chi_q(G) \leq \inf_{\epsilon >0} \Psi_{\epsilon}(G)$.
In the recent work~\cite{BLP}, a different approach was investigated, based on  the construction of a  hierarchy of polyhedral cones 
 that covers the interior of the cone $\cgpsd$.
This can be used to show the existence of a linear program expressing the quantum parameter   $\wtilde \chi_q(G)$ (after rounding).
We hope that these results will stimulate further research leading to a better understanding of the quantum graph parameters.

\medskip
We believe that the cone $\cgpsd$ is an intrinsically very interesting cone, whose structure deserves to be better understood. To conclude we mention another   interesting problem about this cone: given a matrix $A \in \cgpsd$, find upper bounds on the smallest dimension $d$ of the matrices forming a Gram representation of $A$. This corresponds to giving an upper bound on the amount of entanglement needed to perform certain protocols \cite{Cleve:2004} and to finding low dimensional factorizations of nonnegative matrices \cite{Fiorini,GPT}, which are currently attracting much attention.

\section*{Acknowledgments}
T.P. was partially funded by the SIQS European project.
We thank three anonymous referees for their useful remarks that helped improve the presentation of the paper.
In particular, we are grateful to one referee for suggesting the result of Lemma \ref{lemthetap} which permitted to correct an error in the proof of Proposition \ref{lem:theta'equiv}.
The authors also  thank Jop Bri\"et, Sabine Burgdorf, Peter Dickinson, Hamza Fawzi and Markus Schweighofer for useful discussions.

\end{document}